\documentclass[a4paper,10pt]{article}
\usepackage[utf8]{inputenc}
\usepackage{datetime}
\usepackage{centernot}

\newdateformat{monthyeardate}{%
  \monthname[\THEMONTH], \THEYEAR}

\usepackage{amsmath,amsthm,amssymb,amsopn}

\newbox\gnBoxA
\newdimen\gnCornerHgt
\setbox\gnBoxA=\hbox{$\ulcorner$}
\global\gnCornerHgt=\ht\gnBoxA
\newdimen\gnArgHgt
\def\gnmb #1{%
\setbox\gnBoxA=\hbox{$#1$}%
\gnArgHgt=\ht\gnBoxA%
\ifnum     \gnArgHgt<\gnCornerHgt \gnArgHgt=0pt%
\else \advance \gnArgHgt by -\gnCornerHgt%
\fi \raise\gnArgHgt\hbox{$\ulcorner$} \box\gnBoxA %
\raise\gnArgHgt\hbox{$\urcorner$}}

\theoremstyle{plain}
\newtheorem{thm}{Theorem}[section]
\newtheorem{cor}[thm]{Corollary}
\newtheorem{lem}[thm]{Lemma}

\newtheorem{prop}[thm]{Proposition}

\theoremstyle{definition}

\newtheorem{rem}[thm]{Remark}

\newcommand{\defiff}{\stackrel{\mbox{\scriptsize $\textrm{def}$}}{\iff}}

\newcommand{\ain}{\mathrel{\varepsilon}}
\newcommand{\nain}{\mathrel{\centernot\varepsilon}}

\begin{document}
\title{A weak set theory that proves its own consistency}
\author{Fedor~Pakhomov${}^1{}{}^2$}
\date{{\small
${}^1$Steklov Mathematical Institute of Russian Academy of Sciences, Moscow \\ 
${}^2$Institute of Mathematics of the Czech Academy of Sciences, Prague\\}
\vspace{10pt}
August 2019}

\maketitle

\begin{abstract}
In the paper we introduce a  weak set theory $\mathsf{H}_{<\omega}$.
A formalization of arithmetic on finite von Neumann ordinals gives an embedding of arithmetical language into this theory. We show that $\mathsf{H}_{<\omega}$ proves a natural arithmetization of its own Hilbert-style consistency. Unlike some previous examples of theories proving their own consistency, $\mathsf{H}_{<\omega}$ appears to be sufficiently natural.

 The theory $\mathsf{H}_{<\omega}$ is infinitely axiomatizable and proves existence of all individual hereditarily finite sets, but at the same time all its finite subtheories have finite models. Therefore, our example avoids the strong version of G{\"o}del's second incompleteness theorem (due to Pudl{\'a}k) that asserts that no consistent theory interpreting Robinson's arithmetic $\mathsf{Q}$ proves its own consistency \cite{Pud85}. To show that $\mathsf{H}_{<\omega}$ proves its own consistency we establish a conservation result connecting Kalmar elementary arithmetic $\mathsf{EA}$ and $\mathsf{H}_{<\omega}$.
 
 We also consider the version of $\mathsf{H}_{<\omega}$ over higher order logic denoted $\mathsf{H}^{\omega}_{<\omega}$. It has the same ``non-G{\"o}delian'' property as $\mathsf{H}_{<\omega}$ but happens to be more attractive from a technical point of view. In particular, we show that $\mathsf{H}^{\omega}_{<\omega}$ proves a $\Pi_1$ sentence $\varphi$ of the predicate-only version of arithmetical language iff $\mathsf{EA}$ proves that $\varphi$ holds on the superexponential cut.
\end{abstract}

\section{Introduction} G{\"o}del's incompleteness theorems are among the most outstanding results in mathematical logic. The present paper is about the limits of applicability of G{\"o}del's second incompleteness theorem (G2). The standard non-precise formulation of G2 is that no strong enough formal system could prove its own consistency. However in order to make this formulation a mathematical theorem it is necessary to specify exact mathematical meaning to the terms used in it.

Kurt G\"odel in \cite{God31} has considered theories that extend the system $\mathsf{P}$ (a variant of Principia Mathematica system) by primitive recursive sets of axioms. He have developed certain formalization of the consistency assertion for theories of this class within the language of system $\mathsf{P}$. And has proved that no consistent theory from this class could prove the formalization of its own consistency. Here it is important to note that although G\"odel used higher-order system $\mathsf{P}$, he have not relied on the higher-order features of system $\mathsf{P}$ and essentially the same construction would work with first-order arithmetic $\mathsf{PA}$ as the base system.

Further we will discussing more general forms of G2 and examples of theories for which it fails. Our focus will be in generalizations that still talk about unprovability of formalized (in some sense) consistency assertion. There are  generalizations of G2 that does not fall into this category but rather are formulated in terms of interpretability \cite{Fef60,Vis08}. 

One approach to define what is a formalized consistency assertion for a theory $T$ is just to fix some canonical way of producing sentences $\mathsf{Con}(T)$ from an axiomatization of $T$. Typically this is done in the language of first arithmetic\footnote{We note that the usage of first-order arithmetical language here is primarily due to the fact that it is the standard approach in the field and in principle it is possible to formalize consistency statement in any other first-order language of the same of higher expressive power, e.g. binary strings with concatenation, set-theory.}. This allows to ask whether a theory $T$ proves $\mathsf{Con}(T)$, provided that there is a fixed embedding of arithmetical language into the language of $T$. In the present paper we give an examples of theories that prove their own consistency in this sense. 

The other approach is to consider some axiomatic conditions on what constitute a suitable formalization of the consistency assertion. The generalization of G2 developed by David Hilbert and Paul Bernays \cite{BerHil39} was based on the conditions on the formalized proof predicate, i.e. a predicate $\mathsf{Prf}_T(x,y)$ which intended meaning is that $x$ is a code of $T$-proof of formula with G{\"o}del number $y$. Latter the approach was simplified by Martin Hugo L{\"o}b \cite{Lob55}, who instead has formulated what is now known as Hilbert-Bernays-L\"ob derivability conditions (HBL conditions) on predicate $\mathsf{Prv}_T(x)$:
\begin{enumerate}
    \item $T\vdash \varphi\;\Rightarrow\;T\vdash\mathsf{Prv}_T(\gnmb{\varphi})$, for all sentences $\varphi$;
    \item $T\vdash \mathsf{Prv}_T(\gnmb{\varphi\to\psi})\to(\mathsf{Prv}_T(\gnmb{\varphi})\to \mathsf{Prv}_T(\gnmb{\psi}))$, for all sentences $\varphi,\psi$;
    \item $T\vdash\mathsf{Prv}_T(\gnmb{\varphi})\to \mathsf{Prv}_T(\gnmb{\mathsf{Prv}_T(\gnmb{\varphi})})$, for all sentences $\varphi$.
\end{enumerate}
The intended meaning of $\mathsf{Prv}_T(x)$ is that formula with G{\"o}del number $x$ is provable in $T$. By a standard technique, it is possible to prove that if $T$ is a consistent extension of Robinson's arithmetic  $\mathsf{Q}$ ($\mathsf{PA}$ without induction) and $\mathsf{Prv}_T(x)$ satisfies HBL-conditions, then $T$ does not prove the corresponding  consistency assertion $\lnot\mathsf{Prv}_T(\gnmb{0=1})$ (see Petr H{\'a}jek and Pavel Pudl{\'a}k book \cite[Theorem~III.2.21]{HajPud17}). In fact, in addition to HBL-conditions, the proof of the result only uses first-order reasoning in $T$ and existence of sentence $\varphi$ such that  $T\vdash \varphi\mathrel{\leftrightarrow} \lnot \Box \varphi$ (which is produced by Diagonal Lemma). Since even weaker Robinson's arithmetic $\mathsf{R}$ proves Diagonal Lemma \cite[Theorem~3]{Sve07} the argument works for its extensions as well (see Appendix \ref{interpretability} for definition of $\mathsf{R}$). There are very general results about unprovability of consistency under HBL-conditions and either presence of formalization of syntax (by Robert Jeroslow \cite{Jer73}) or presence of the appropriate fixed points (by Lev Beklemishev and Daniyar Shamkanov \cite{BekSha16}). Beklemishev and Shamkanov \cite{BekSha16} provided an example of a system based on a certain contraction-free logic (instead of usually used classical or intuitonistic logic), where their abstract version of G2 fails.

Further we will focus on provability/unprovability of natural formalizations of the consistency assertion rather than formalizations arising from arbitrary predicates satisfying HBL-conditions. Solomon Feferman \cite{Fef60} fixed an arithmetization of (Hilbert-style) provability in first-order logic. This allowed him to produce formalized consistency assertion $\mathsf{Con}(T)$ given formula $\mathsf{Ax}_T(x)$ defining the set of axiom of $T$; further, for naturally chosen formula $\mathsf{Ax}_T(x)$ we will call the sentente $\mathsf{Con}(T)$ the Hilbert-style consistency of $T$. Strong enough consistent c.e. theories are capable to show that their Hilbert-style provability predicate satisfies HBL conditions and thus they could not prove their own Hilbert-style consistency. In \cite{Fef60} Feferman considered extensions of $\mathsf{PA}$, but it works even for c.e. extensions of much weaker system $\mathsf{EA}+\mathsf{B}\Sigma_1$. 

Although HBL conditions do not necessary hold in the case of some weaker arithmetical c.e. theories, in many cases it is still possible to establish unprovability of Hilbert-style consistency for them.  Namely Pudl{\'a}k \cite{Pud85} proved that any consistent c.e. arithmetical theory extending Robinson's arithmetic $\mathsf{Q}$ could not prove its own Hilbert-style consistency statement. The essential part of the Pudl{\'a}k's argument was to show that a failure of G2 for a theory $T$, where HBL conditions might not be satisfied, leads to a failure of G2 in a different theory $T'$ (interpretable in $T$), where HBL are satisfied.

However, here it is crucial that one considers c.e. axiomatization of a theory. As it have been already observed by Feferman \cite{Fef60}, for certain non-$\Sigma_1$ formula defining the set of axioms of $\mathsf{PA}$ the theory $\mathsf{PA}$ could prove the respective consistency statement. Another interesting example of similar sort have been provided by Karl-Georg Niebergall \cite{Nie05} who showed that the theory $(\mathsf{PA}+\mathsf{RFN}(\mathsf{PA}))\cap (\mathsf{PA}+\mbox{ all true $\Pi_1$ sentences})$ could prove its own natural consistency sentence.

Pudl{\'a}k's result mentioned above could be generalize further to arithmetical theories in predicate-only signature that could prove totality of successor function (the totality of successor function is important since it is necessary for the cut-shortening technique employed by Pudl{\'a}k). However, Dan Willard has constructed examples of c.e. arithmetical theories that could not prove the totality of successor function but could prove their own Hilbert-style consistency \cite{Wil01,Wil06}. The theories in his examples are not completely natural in the sense that some of axioms are constructed using Diagonal Lemma. The main result of the present paper is the construction of a more natural example of this kind.

We define a theory $\mathsf{H}_{<\omega}$ and show that it proves its own Hilbert-style consistency. The language of system $\mathsf{H}_{\omega}$ set theory with additional unary function $\overline{\mathsf{V}}$. 

First we define a weaker system $\mathsf{H}$ with the following axioms.
\begin{enumerate}
    \item  $x=y \mathrel{\leftrightarrow}\forall z(z\in x\mathrel{\leftrightarrow} z\in y)$ (Extensionality);
\item $\exists y \forall z(z\in y\mathrel{\leftrightarrow}z\in x \land \varphi(z))$, where $\varphi(z)$ ranges over first-order formulas without free occurrences of $y$ (Separation); 
\item $y\in \overline{\mathsf{V}}(x)\mathrel{\leftrightarrow}(\exists z\in x)(y\subseteq \overline{\mathsf{V}}(z))$ (Defining axiom for $\overline{\mathsf{V}}$).
\end{enumerate}
Here the function $\overline{\mathsf{V}}$ is intended to be the function that maps a set $x$ to the least level $\mathsf{V}_{\alpha}$ of von Neumann hierarchy such that $x\subseteq \mathsf{V}_\alpha$; note that the defining axiom for $\overline{\mathsf{V}}$ essentially states that $\overline{\mathsf{V}}$ satisfies the following recursive definition:
$$\overline{\mathsf{V}}(x)=\bigcup\limits_{y\in x} \mathcal{P}(\overline{\mathsf{V}}(y)).$$
And the theory $\mathsf{H}_{<\omega}$ is defined to be the extension of $\mathsf{H}$ by all the axioms $\exists x\;\mathsf{Nmb}_{n}(x)$ stating the existence of all individual finite von Neumann ordinals $n$.

The theory $\mathsf{H}$ is an incomplete first-order theory which intended models are arbitrary levels of von Neumann hierarchy $(\mathsf{V}_\alpha,\in,\overline{\mathsf{V}})$. To motivate this let us look at the second-order version $\mathsf{H}^2$ of the system $\mathsf{H}$ ($\mathsf{H}^2$ is $\mathsf{H}$ with the scheme of separation extended to the second-order language).  It is easy to show that models of $\mathsf{H}^2$ with the standard second-order part are up to isomorphism the models $(\mathsf{V}_\alpha,\in,\overline{\mathsf{V}})$. And the additional axioms of $\mathsf{H}_{<\omega}$ rule out the finite levels of von-Neumann hierarchy as potential models leaving only the models that contain all the individual finite sets.

Theory $\mathsf{H}$ is capable of formalization of ordinal arithmetic in a relatively standard manner. Important restriction here is that both $\mathsf{H}$ and $\mathsf{H}_{<\omega}$  could not prove totality of successor function. The restriction of the ordinal arithmetic to the finite ordinals gives an interpretation of predicate version of arithmetical language in $\mathsf{H}$. Observe that $\mathsf{H}$ could meaningfully talk about $\Pi_1$ sentences of predicate-only version of arithmetical language (we denote this class as $\Pi_1^{\mathsf{pred}}$). In a standard manner we transform the usual Hilbert-style consistency sentence $\mathsf{Con}(T)$  (that is $\Pi_1$ is standard arithmetical language) to a predicate-only $\Pi_1^{\mathsf{pred}}$ sentence $\mathsf{Con}^{\mathsf{pred}}(T)$. Our main result is that the theory $\mathsf{H}_{<\omega}$ (and actually even $\mathsf{H}$) proves $\mathsf{Con}^{\mathsf{pred}}(\mathsf{H}_{<\omega})$. We note that our result is fairly robust with respect to the choice of the particular arithmetization of the notion of proof, which determines the exact form of the sentence $\mathsf{Con}^{\mathsf{pred}}(T)$ (see discussion in the beginning of Section \ref{Non-god_section}).  

For technical reasons it happens to be easier to establish the analogous result for the higher-order version $\mathsf{H}^{\omega}$ of the system $\mathsf{H}$. Actually the system $\mathsf{H}^{\omega}$ allows for more natural reasoning than $\mathsf{H}$. Both $\mathsf{H}$ and $\mathsf{H}^{\omega}$ are consistent with the situation when there is the maximal well-founded rank $\alpha_0$ of sets. And when one reasons about the sets which rank is close to $\alpha_0$ in $\mathsf{H}$ there might be sever restrictions on the type of set-theoretic construction it is possible to perform with the set. For example, there could be sets $x,y$ such that their Cartesian product does not exist. Within the system $\mathsf{H}^{\omega}$ this kind of problems could be addressed; for example, even if for two sets $x,y$ their Cartesian product is not a set, still $x\times y$ could be represented by a higher order object. Thus we find $\mathsf{H}^{\omega}_{<\omega}$ to be even more interesting example of non-G{\"o}delian theory than $\mathsf{H}_{<\omega}$ itself. Since the case of theory $\mathsf{H}^{\omega}$ is technically simpler, we first study it is Section \ref{H_omega_section}--\ref{Non-god_section}. And consider the case of the theory $\mathsf{H}$ only latter in Section \ref{H_appendix}.

Now we give a general idea of how to prove Hilbert-style consistency of $\mathsf{H}^{\omega}_{<\omega}$ in $\mathsf{H}^{\omega}$. Although it is possible to give a direct proof of consistency of  $\mathsf{H}^{\omega}_{<\omega}$ in $\mathsf{H}^{\omega}$, in the paper we obtain the result by proving suitable conservativity theorems.  Our main technical result is Lemma \ref{HCST_EA_ON}: $\mathsf{H}^{\omega}$ proves a $\Pi_1^{\mathsf{pred}}$ sentence $\varphi$ iff $\mathsf{EA}$ proves $\varphi^{\mathcal{S}}$. Here $\varphi^{\mathcal{S}}$ is the relativization of $\varphi$ to the superexponential cut $\mathcal{S}$ that consists of all natural numbers $x$ for which the superexponentiation $2^0_{x}$ is defined;  by definition we put $2^y_0=y$ and $2^y_{x+1}=2^{2^y_x}$. From Lemma \ref{HCST_EA_ON} it follows that in order to prove $\mathsf{Con}^{\mathsf{pred}}(\mathsf{H}^{\omega}_{<\omega})$ in $\mathsf{H}^{\omega}$ it is enough to prove in $\mathsf{EA}$ that for any G\"odel number $p$ of a $\mathsf{H}^{\omega}_{<\omega}$-proof, if $2^0_p$ is defined then $p$ could not be a proof of contradiction in $\mathsf{H}^{\omega}_{<\omega}$. To prove the latter inside $\mathsf{EA}$ we construct a finite model $\mathcal{M}$ of the size $\le 2^0_p$ that satisfy all the axioms of $\mathsf{H}^{\omega}_{<\omega}$ that occur in $p$.

To prove the mentioned conservation result between $\mathsf{H}^{\omega}$ and $\mathsf{EA}$ (Lemma \ref{HCST_EA_ON}), we introduce a theory $\mathsf{EA}^{\mathsf{set}}$ that is a theory in the language with the predicate $\in$, unary function $\mathcal{P}$, and unary function $\overline{\mathsf{V}}$. 

On one hand, we prove that that the theory $\mathsf{H}^{\omega}$ proves the same $\Pi_1^{\mathsf{set}}(\overline{\mathsf{V}})$ set-theoretic sentences about hereditarily finite sets as the theory $\mathsf{EA}^{\mathsf{set}}$. On the other hand, we show that there is a natural bi-interpretation between $\mathsf{EA}$ and $\mathsf{EA}^{\mathsf{set}}$.  The bi-interpretation is formed by the cardinal arithmetic interpretation of $\mathsf{EA}$ in $\mathsf{EA}^{\mathsf{set}}$ together with the interpretation of $\mathsf{EA}^{\mathsf{set}}$ in $\mathsf{EA}$ by Ackermann's membership predicate (see Section \ref{bi-interpretabilit_section}). Additionally we show that under this bi-interpretation the ordinal arithmetic in the theory $\mathsf{EA}^{\mathsf{set}}$ correspond to the arithmetic on numbers from superexponential cut in $\mathsf{EA}$. Combination of this three facts allows us to prove Lemma \ref{HCST_EA_ON}.

We note that bi-interpretability between $\mathsf{EA}$ and $\mathsf{EA}^{\mathsf{set}}$ is a modified version of a result by Richard Pettigrew \cite{Pet09}, who have introduced a set theory $\mathsf{EA}^{\star}$ and proved that it is bi-interpretable with $\mathsf{EA}$ (see Section \ref{Theory_EA^set} for more extensive discussion). 

\section{Theories $\mathsf{H}^{\omega}$ and $\mathsf{H}^{\omega}_{<\omega}$}\label{H_omega_section}
\subsection{Higher order logic}
We start with the description of the version of higher order logic that we use for the theories $\mathsf{H}^{\omega}$ and $\mathsf{H}^{\omega}_{<\omega}$.

Types are indexed by natural numbers. The type $0$ is the type of individual objects and the type $n+1$ is the type of sets of the objects of the type $n$. We have quantifiers over objects of any type. We allow comprehension over arbitrary properties expressible in higher order logic.

Formally, we define the deductive system for higher order logic on top of deductive system for many sorted first-order logic with equality, where types are indexed by natural numbers. We have membership predicates $x^{(n)}\ain y^{(n+1)}$ between object $x^{(n)}$ of the type $n$ and object $y^{(n+1)}$ of the type $n+1$. As additional principles we have extensionality axioms
$$\forall z^{(n)}\;(z^{(n)}\ain x^{(n+1)}\mathrel{\leftrightarrow}z^{(n)}\ain y^{(n+1)})\to x^{(n+1)}=y^{(n+1)}$$
and comprehension schemes
$$\exists x^{(n+1)}\forall y^{(n)}\;(y^{(n)}\ain x^{(n+1)}\mathrel{\leftrightarrow}\varphi(y^{(n)})),$$
where $\varphi$ is any higher order formula without free occurrences of $x^{(n+1)}$.

We could develop some standard constructions in the higher order logic. First let us define representations for ordered pairs $\langle x^{(n)},y^{(m)}\rangle$. A pair  $\langle x^{(n)},y^{(n)}\rangle$ is encoded by the set $\{\{x^{(n)}\},\{x^{(n)},y^{(n)}\}\}$ of the type $n+2$. For $m<n$ the pairs $\langle x^{(n)},y^{(m)}\rangle$ are  encoded by the pairs $\langle x^{(n)},\mathop{\underbrace{\{\ldots\{}y^{(m)}\underbrace{\}\ldots\}}}\limits_{\mbox{\footnotesize $n-m$ times nested}}\rangle$. And for $m>n$ the pairs  $\langle x^{(n)},y^{(m)}\rangle$ are encoded by the pairs $\langle \mathop{\underbrace{\{\ldots\{}x^{(n)}\underbrace{\}\ldots\}}}\limits_{\mbox{\footnotesize $m-n$ times nested}},y^{(m)}\rangle$. For sets $X^{(n+1)}$ and $Y^{(m+1)}$ we denote as $X^{(n+1)}\times Y^{(m+1)}$ their Cartesian product $\{\langle x^{(n)},y^{(m)}\rangle \mid x^{(n)}\in X^{(n+1)}, y^{(m)}\in Y^{(m+1)}\}$. We encode a function $f\colon  X^{(n+1)}\to Y^{({m+1})}$  by its graph, i.e. by the set of pairs  $$\{\langle x^{(n)}, y^{(m)}\rangle \mid x^{(n)}\ain X^{(n+1)}\mbox{ and }y^{(m)}=f(x^{(n)})\}.$$  In the same way we could encode partial functions. 
For a set $x^{(n)}$ we denote by $\mathcal{P}(x^{(n)})$ the type $n+1$ set that consists of all the sets $y^{(n)}$ that are subsets of $x^{(n)}$.


\subsection{Theories  $\mathsf{H}^{\omega}$ and $\mathsf{H}^{\omega}_{<\omega}$} \label{Higher_order_theory}
Theories $\mathsf{H}^{\omega}$ and $\mathsf{H}^{\omega}_{<\omega}$ are formulated over higher order logic with the only non-logical symbols being membership predicate $x^{(0)}\in y^{(0)}$ and function $\overline{\mathsf{V}}(x^{(0)})$ with the values of the type $0$.

We use some additional naming convention. We call objects of the  type $0$ sets and use small Latin letters $x,y,z,\ldots$ for variables over sets. We call object of type $n$ type $n$ sets. Also we call type $1$ sets classes and use capital Latin letters $X,Y,Z,\ldots$ without upper indexes for variables that range over classes. We denote the class of all sets as $\mathsf{V}$. We use $\subseteq$ as the standard shorthand both for sets of type $0$ and of higher types: the expression $x\subseteq y$ is the shorthand for $\forall z\; (z\in x\to z\in y)$, the expression $x^{(n+1)}\subseteq y^{(n+1)}$ is the shorthand for  $\forall z^{(n)}\; (z^{(n)}\ain x^{(n+1)}\to z^{(n)}\ain y^{(n+1)})$, the expression $x\subseteq Y$ is the shorthand for $\forall z\;(z\in x\to z\ain Y)$, and the expression $X\subseteq y$ is the shorthand for $\forall z\;(z\in X\to z\ain y)$.

The axioms of $\mathsf{H}^{\omega}$ are:
\begin{enumerate}
\item $x=y \mathrel{\leftrightarrow}\forall z(z\in x\mathrel{\leftrightarrow} z\in y) $ (Extensionality);
\item $\exists z \forall w(w\in z\mathrel{\leftrightarrow}w\in x \land w\ain Y)$ (Separation); 
\item \label{V_existence_axiom} $ y\in \overline{\mathsf{V}}(x)\mathrel{\leftrightarrow}(\exists z\in x)(y\subseteq \overline{\mathsf{V}}(z))$ (Defining Axiom for $\overline{\mathsf{V}}$).
\end{enumerate}

\begin{lem} \label{foundation_in_H^omega} Theory $\mathsf{H}^{\omega}$ proves the axiom of $\varepsilon$-induction $$\forall X(\forall y((\forall z\in y) z\ain X\to y\ain X)\to \forall y ( y\ain X)).$$
\end{lem}
\begin{proof} We say that a class $X$ is progressive if $\forall x ((\forall y\in x)\;y\ain X\to x\ain X)$. We define the class $\mathsf{WF}$ to be the intersection of all  progressive classes. The class $\mathsf{WF}$ is progressive itself. Clearly, if $\mathsf{WF}=\mathsf{V}$, then we are done.

Assume for a contradiction that $\exists x\; (x\nain \mathsf{WF})$. Then we consider the class $$A=\{x\ain \mathsf{WF}\mid \forall y \nain \mathsf{WF} \; (x\in \overline{\mathsf{V}}(y))\}.$$
In other words, $A$ is the intersection $\bigcap\limits_{y\nain \mathsf{WF}}\overline{\mathsf{V}}(y)$.First we prove that $A$ is progressive. We consider some $x\subseteq A$ and claim that $x\in A$. For this we consider arbitrary $y\nain \mathsf{WF}$ and prove that $x\in \overline{\mathsf{V}}(y)$. By progressivity of $\mathsf{WF}$ there is $z\nain \mathsf{WF}$ such that $z\in y$. Since $x\subseteq A$, by definition of $A$, we have $x\subseteq \overline{\mathsf{V}}(z)$. By definition of $\overline{\mathsf{V}}$, the set $\overline{\mathsf{V}}(y)$ contains any subset of $\overline{\mathsf{V}}(z)$ and thus our claim $x\in \overline{\mathsf{V}}(y)$ holds. 

Thus $\mathsf{WF}\subseteq A \subseteq \overline{\mathsf{V}}(x)$, for any set $x\nain \mathsf{WF}$. We fix any $x_0\nain \mathsf{WF}$ and by separation construct the set $\mathsf{wf}=\{x\in \overline{\mathsf{V}}(x_0)\mid x\ain \mathsf{WF}\}$. Clearly, $\mathsf{wf}$ consists of exactly the same elements as $\mathsf{WF}$ and in particular  $\mathsf{wf}\subseteq \mathsf{WF}$. Since $\mathsf{wf}$ is a set,  progressivity of $\mathsf{WF}$ implies that $\mathsf{wf}\ain \mathsf{WF}$. Hence $\mathsf{wf}\in \mathsf{wf}$. Observe that the class $\{x\mid x\not \in x\}$ is progressive, hence $\mathsf{wf}$ is an element of this class. Thus $\mathsf{wf}\not \in \mathsf{wf}$, contradiction. \end{proof}

The theory $\mathsf{H}^{\omega}_{<\omega}$ is the extension of $\mathsf{H}^{\omega}$ by the axioms stating the existence of all individual finite von Neumann ordinals. Let $\mathsf{Nmb}_0(x)$ be the formula $\forall y\;(y\not\in x)$ and let $\mathsf{Nmb}_{n+1}(x)$ be the formulas $\exists y\;(\mathsf{Nmb}_{n}(y)\land \forall z(z\in x\mathrel{\leftrightarrow} z=y\lor z\in y))$. Theory $\mathsf{H}^{\omega}_{<\omega}$ is the extension of $\mathsf{H}^{\omega}$ by the axioms $\exists x\;\mathsf{Nmb}_n(x)$, for all natural numbers $n$.

\subsection{Ordinal arithmetic in $\mathsf{H}^{\omega}$}\label{ordinal_arithmetic_in_H_omega}
Due to the fact that the class of intended models of $\mathsf{H}^{\omega}$ are the models $(\mathsf{V}_{\alpha},\in,\overline{\mathsf{V}})$, for ordinals $\alpha>0$, a lot of functions on sets that are total in stronger set theories could not be proved to be total in $\mathsf{H}^{\omega}$. Thus we need to work with partial functions. 

We make the following definitions inside $\mathsf{H}^{\omega}$:
\begin{enumerate}
\item class of transitive sets $\mathsf{Trans}=\{x\mid \forall y\in x\; y\subseteq x\}$;
\item class of ordinals $\mathsf{On}=\{x\mid x\ain \mathsf{Trans}\mbox{ and } (\forall y\in x)\; y\ain \mathsf{Trans} \}$;
\item the order $<$ on ordinals is given by the predicate $\in$;
\item the ordinal $0=\emptyset$ (it is defined only if $\mathsf{V}$ is non-empty);
\item the partial successor function $S\colon \mathsf{On}\to \mathsf{On}$ 
$$\beta=S(\alpha)\defiff \forall x (x\in \beta \mathrel{\leftrightarrow} x \in \alpha \lor x=\alpha);$$
\item the class of successor ordinals $\mathsf{Succ}=\{\alpha \mid (\exists \beta \ain \mathsf{On})\; \alpha= S(\beta)\}$;
\item the class of natural numbers $\mathsf{Nat}=\{\alpha \mid \forall \beta\le \alpha (\beta\in \mathsf{Succ}\lor \beta=0)\}$.
\item \label{add_def} the partial function $+\colon \mathsf{On}\times \mathsf{On}\to \mathsf{On}$ that is the only partial function such that for all $\alpha,\beta\ain \mathsf{On}$
$$\alpha+\beta=\sup(\{\alpha\}\cup \{S(\alpha+\gamma)\mid \gamma<\beta\}),$$
where the left part is defined whenever the right part is defined, i.e. the values $\alpha+\beta$ should be defined iff the values $S(\alpha+\gamma)$ are defined for all $\gamma<\beta$ and the class $\{\alpha\}\cup \{S(\alpha+\gamma)\mid \gamma<\beta\}$ have a supremum;
\item the partial function $\cdot\colon \mathsf{On}\times \mathsf{On}\to \mathsf{On}$ is the only partial function such that for all $\alpha,\beta\ain \mathsf{On}$
$$\alpha \mathop{\cdot}\beta=\sup(\{\alpha\mathop{\cdot}\gamma+\alpha \mid \gamma<\beta\}),$$
where the left part is defined whenever the right part is defined;
\item the partial base $2$ exponentiation $2^{x}\colon \mathsf{On}\to \mathsf{On}$ is the only partial function such that for all $\alpha,\beta\ain \mathsf{On}$
$$2^\alpha=\sup(\{S(0)\}\cup\{2^\beta+2^\beta \mid \beta<\alpha\}),$$
where the left part is defined whenever the right part is defined.
\end{enumerate}

We note that the existence and uniqueness of addition, multiplication, and base 2 exponentiation functions could be proved in a standard fashion. We consider only the case of addition, since the cases of the rest of the functions could be covered in the same manner. We call a partial function $+'\colon \mathsf{On}\times \mathsf{On}\to\mathsf{On}$ a partial addition function if for any ordinals $\alpha,\beta$ the fact that $\alpha\mathop{+'}\beta$ is defined implies that $\sup(\{\alpha\}\cup \{S(\alpha\mathop{+'}\gamma)\mid \gamma<\beta\})$ is defined (along side with all $S(\alpha\mathop{+'}\gamma)$, for $\gamma<\beta$) and equal to $\alpha\mathop{+'}\beta$. Using $\varepsilon$-induction it is easy to show that any two partial addition functions agree on the pairs of ordinals where they both are defined. We observe that the union of all the partial addition functions constitute an addition function that satisfies the definition above. Finally, we use $\varepsilon$-induction to prove uniqueness of partial addition function that satisfies \ref{add_def}.


Let us consider predicate-only version of arithmetical language, where we have the predicates $x=y$, $x\le y$, $x=S(y)$, $x=y+z$, $x=yz$, and $x=2^y$ (see \cite[Section~I.2]{HajPud17}). Our definition of the partial arithmetical function on natural numbers in $\mathsf{H}^{\omega}$  gives an interpretation $\mathcal{NAT}$ of this version of arithmetical language in the theory $\mathsf{H}^{\omega}$.



\section{Theory $\mathsf{EA}^{\mathsf{set}}$} \label{Theory_EA^set}
In the section we will develop a set theory $\mathsf{EA}^{\mathsf{set}}$ that 1. is bi-interpretable with $\mathsf{EA}$ and 2. proves that same $\Pi_1^{\mathsf{set}}(\overline{\mathsf{V}})$ sentences of first-order pure set-theoretic language as $\mathsf{H}^{\omega}$ proves for hereditarily finite sets.

We note that R. Pettigrew  \cite{Pet09} already have proposed set theory theory $\mathsf{EA}^{\star}$ that is bi-interpretable with $\mathsf{EA}$. The language of Pettigrew's  theory $\mathsf{EA}^{\star}$ is the language of pure first-order set theory,  unlike the language of our theory $\mathsf{EA}^{\mathsf{set}}$ that in addition uses functions $\overline{\mathsf{V}}(x)$ and $\mathcal{P}(x)$. Using both our and Pettigrew's results about bi-interpretability, it is easy to show that the theory $\mathsf{EA}^{\mathsf{set}}$ is just a definitional extension of the theory $\mathsf{EA}^{\star}$.

Let us outline the main differences between our and Pettigrew's approaches. The first difference is that due to the richer signature, the axiomatization of $\mathsf{EA}^{\mathsf{set}}$ is simpler than the axiomatization of $\mathsf{EA}^{\star}$. Also there is a difference between the constructed bi-interpretations. The bi-interpretation that we define consists of two natural interpretations: Ackermann's interpretation of set theory in arithmetic and cardinal interpretation of arithmetic in set theory. Pettigrew's bi-interpretation consisted of Ackermann's interpretation of set theory in arithmetic and certain somewhat artificial interpretation of arithmetic in set theory. However, Pettigrew's interpretations have the advantage of being strictly inverse to each other. Wheres our interpretations are not strictly inverse to each other: the compositions of Ackermann's and cardinality interpretations are self-interpretations of $\mathsf{EA}^{\star}$ and of $\mathsf{EA}$ that are not identity interpretations themselves, but rather are definably isomorphic to identity interpretations. 

We note that it was possible to use Pettigrew's result to somewhat shorten the paper. However, in order to make the presentation in the present paper more self-sufficient and to make proofs more direct we will not rely on Pettigrew's paper.

The language of the theory $\mathsf{EA}^{\mathsf{set}}$ is the language of first-order set theory expanded by the unary functions  $\overline{\mathsf{V}}(x)$ and $\mathcal{P}(x)$. We denote by $\Delta_0^{\mathsf{set}}(\mathcal{P},\overline{\mathsf{V}})$ the class of first-order formulas, where all the quantifiers are of the form $(\forall x\in t)$ or $(\exists y\in t)$, for some term $t$ built of the functions  $\mathcal{P}$, $\overline{\mathsf{V}}$ and the variables other than $x$. And we denote as $\Pi_1^{\mathsf{set}}(\mathcal{P},\overline{\mathsf{V}})$ the class of formulas $\forall \vec{x}\;\varphi$, where $\varphi\in\Delta_0^{\mathsf{set}}(\mathcal{P},\overline{\mathsf{V}})$. The expression $x\subseteq y$ is a shorthand for $\forall z\;(z\in x\to z\in y)$.

The theory $\mathsf{EA}^{\mathsf{set}}$ is axiomatized over the usual first-order logic with equality by the following axioms. 
\begin{enumerate}
    \item \label{EA^set_1} $x=y \mathrel{\leftrightarrow} \forall z(z\in x\mathrel{\leftrightarrow} z\in y))$ (Extensionality);
    \item \label{EA^set_2} $\exists y\forall z\;(z\in y \mathrel{\leftrightarrow} z\in x\land \varphi(z))$, where $\varphi$ is $\Delta_0^{\mathsf{set}}(\mathcal{P},\overline{\mathsf{V}})$ formula without free occurrences of $y$ ($\Delta_0^{\mathsf{set}}(\mathcal{P},\overline{\mathsf{V}})$-Separation);
    \item \label{EA^set_3} $y\in \mathcal{P}(x)\mathrel{\leftrightarrow} y\subseteq x$ (Defining Axiom for $\mathcal{P}$);
    \item \label{EA^set_4} $y\in \overline{\mathsf{V}}(x)\mathrel{\leftrightarrow} (\exists z\in x) y\in \mathcal{P}(\overline{\mathsf{V}}(z))$ (Defining Axiom for $\overline{\mathsf{V}}$);
\item \label{EA^set_5} $\varphi(\emptyset)\land \forall x,y\;( \varphi(x)\land \varphi(y)\to \varphi(x\cup\{y\}))\to \forall x\; \varphi(x)$, where $\varphi$ range over $\Delta_0^{\mathsf{set}}(\mathcal{P},\overline{\mathsf{V}})$ ($\Delta_0^{\mathsf{set}}(\mathcal{P},\overline{\mathsf{V}})$ Adduction Induction)\footnote{This formulation of the scheme \ref{EA^set_5}. is not completely accurate. Namely, the axioms \ref{EA^set_1}.--\ref{EA^set_4}. by themselves are too weak to prove that for any two sets $x,y$ there exists their union $x\cup y$. Thus in \ref{EA^set_5}. the subformula $\varphi(x\cup\{y\})$ should be read as ``if there exists the set $x\cup\{y\}$ then $\varphi(x\cup\{y\})$''. Formally, the scheme \ref{EA^set_5} have the following formulation in the plain language of $\mathsf{EA}^{\mathsf{set}}$: $$\exists x\;( \forall y\;\lnot y\in x \land \varphi(x))\land \forall x,y,z(\varphi(x)\land \varphi(y) \land \forall w(w\in z\mathrel{\leftrightarrow} w\in x\lor w=y)\to \varphi(z))\to \forall x\;\varphi(x).$$
We note that the use of this kind of induction axioms for theories of hereditarily finite sets is going back to the work of Givant and Tarski \cite{GT77}}.
\end{enumerate}
    
First we need to ``bootstrap'' the theory $\mathsf{EA}^{\mathsf{set}}$. 

Observe that scheme \ref{EA^set_5}. (in the presence of the scheme of $\Delta_0^{\mathsf{set}}(\mathcal{P},\overline{\mathsf{V}})$-separation) implies the usual $\varepsilon$-induction scheme over $\Delta_0^{\mathsf{set}}(\mathcal{P},\overline{\mathsf{V}})$ formulas:
$$\forall x ((\forall y\in x)\;\varphi(y)\to \varphi(x))\to \forall x\;\varphi(x),$$
where $\varphi$ ranges over $\Delta_0^{\mathsf{set}}(\mathcal{P},\overline{\mathsf{V}})$-formulas.

Note that our axiomatization of $\mathsf{EA}^{\mathsf{set}}$ does not contain the standard axiom of pair. This is due to the fact that it follows from other axioms of $\mathsf{EA}^{\mathsf{set}}$.  However in order to prove it we first show that both $\mathsf{EA}^{\mathsf{set}}$ and $\mathsf{H}^{\omega}$ prove number of natural properties of $\overline{\mathsf{V}}$. 
\begin{lem} \label{vcl} Theories $\mathsf{EA}^{\mathsf{set}}$ and $\mathsf{H}^{\omega}$ prove that 
\begin{enumerate}
    \item \label{vcl_prop1}$x\subseteq \overline{\mathsf{V}}(x)$, for all $x$;
    \item \label{vcl_prop2}for all $x$ the set $\overline{\mathsf{V}}(x)$ is transitive, i.e.  $(\forall y\in \overline{\mathsf{V}}(x))\; y\subseteq \overline{\mathsf{V}}(x)$;
    \item\label{vcl_prop3} for all $x$ the set $\overline{\mathsf{V}}(x)$ is closed under subsets, i.e.  $(\forall y\in \overline{\mathsf{V}}(x))(\forall z\subseteq y)\; z\in\overline{\mathsf{V}}(x)$;
    \item \label{vcl_prop4}for all $x$ the set $\overline{\mathsf{V}}(x)$ is closed under $\overline{\mathsf{V}}$, i.e. $(\forall y\in \overline{\mathsf{V}}(x))\; \overline{\mathsf{V}}(y)\in \overline{\mathsf{V}}(x)$;
    \item \label{vcl_prop5} $\overline{\mathsf{V}}$ is idempotent, i.e. $\overline{\mathsf{V}}(\overline{\mathsf{V}}(x))=\overline{\mathsf{V}}(x)$, for all $x$;
    \item \label{vcl_prop6}$\overline{\mathsf{V}}(x)\in \overline{\mathsf{V}}(y)$, or $\overline{\mathsf{V}}(x)=\overline{\mathsf{V}}(y)$, or $\overline{\mathsf{V}}(y)\in \overline{\mathsf{V}}(x)$, for all  $x,y$.
\end{enumerate}
\end{lem}
\begin{proof} We will prove Claims \ref{vcl_prop1}--\ref{vcl_prop6} just from extensionality, defining axiom for $\overline{\mathsf{V}}$, and the scheme of $\varepsilon$-induction for $\Delta_0(\overline{\mathsf{V}})$-formulas.

We establish  Claims \ref{vcl_prop1}. and \ref{vcl_prop2}. by straightforward $\varepsilon$-induction arguments on $x$. Claim  \ref{vcl_prop3}. follows from transitivity of $\subseteq$-relation and defining axiom for $\overline{\mathsf{V}}$.

Let us prove \ref{vcl_prop4}. by $\varepsilon$-induction on $x$. To justify the step of induction we need to show that for a given $y\in\overline{\mathsf{V}}(x)$ we have $\overline{\mathsf{V}}(y)\in \overline{\mathsf{V}}(x)$ under the assumption that for all $z\in x$ the sets $\overline{\mathsf{V}}(z)$ are closed under $\overline{\mathsf{V}}$. Indeed, we fix $z\in x$ such that $y\subseteq \overline{\mathsf{V}}(z)$. By induction assumption $\overline{\mathsf{V}}(w)\in z$, for all $w\in y$. Since $\overline{\mathsf{V}}(z)$ is closed under subsets, $\mathcal{P}(\overline{\mathsf{V}}(w))\subseteq \overline{\mathsf{V}}(z)$, for all $w\in y$. Thus $\overline{\mathsf{V}}(y)\subseteq \overline{\mathsf{V}}(z)$. And finally we conclude that $\overline{\mathsf{V}}(y)\in \overline{\mathsf{V}}(x)$. 

Let us prove \ref{vcl_prop5} by showing that $\overline{\mathsf{V}}(\overline{\mathsf{V}}(x))\supseteq \overline{\mathsf{V}}(x)$ and $\overline{\mathsf{V}}(\overline{\mathsf{V}}(x))\subseteq \overline{\mathsf{V}}(x)$. By \ref{vcl_prop1}. we have $\overline{\mathsf{V}}(\overline{\mathsf{V}}(x))\supseteq \overline{\mathsf{V}}(x)$. To show that $\overline{\mathsf{V}}(\overline{\mathsf{V}}(x))\subseteq \overline{\mathsf{V}}(x)$ we consider any $y\in\overline{\mathsf{V}}(\overline{\mathsf{V}}(x))$ and prove that $y\in \overline{\mathsf{V}}(x)$. By defining axiom for $\overline{\mathsf{V}}$ we have $y\subseteq \overline{\mathsf{V}}(z)$, for some $z\in \overline{\mathsf{V}}(x)$. By \ref{vcl_prop4}.  the set $\overline{\mathsf{V}}(z)\in\overline{\mathsf{V}}(x)$. Combining this with \ref{vcl_prop3}. we conclude that $y\in \overline{\mathsf{V}}(x)$, which concludes the proof of \ref{vcl_prop5}.

Finally, let us prove \ref{vcl_prop6}.  For this it is enough to show that for all $x$ and all  $y,z\in \overline{\mathsf{V}}(x)$ we have either $\overline{\mathsf{V}}(y)\in \overline{\mathsf{V}}(z)$, or $\overline{\mathsf{V}}(y)= \overline{\mathsf{V}}(z)$, or $\overline{\mathsf{V}}(z)\in \overline{\mathsf{V}}(y)$. By \ref{vcl_prop5}. it is enough to consider only the case of $x=\overline{\mathsf{V}}(x)$. We prove by $\varepsilon$-induction on $x$ that if $x=\overline{\mathsf{V}}(x)$ then for all  $y,z\in \overline{\mathsf{V}}(x)$ we have either $\overline{\mathsf{V}}(y)\in \overline{\mathsf{V}}(z)$, or $\overline{\mathsf{V}}(y)= \overline{\mathsf{V}}(z)$, or $\overline{\mathsf{V}}(z)\in \overline{\mathsf{V}}(y)$. Further we justify the step of this induction. 

We show that for all $y\in \overline{\mathsf{V}}(x)$ and $z\subseteq y$ either $\overline{\mathsf{V}}(z)=\overline{\mathsf{V}}(y)$ or $\overline{\mathsf{V}}(z)\in \overline{\mathsf{V}}(y)$. From defining axiom for $\overline{\mathsf{V}}$ it follows that $\overline{\mathsf{V}}(z)\subseteq\overline{\mathsf{V}}(y)$. Hence it is enough to prove that if $\overline{\mathsf{V}}(z)\ne\overline{\mathsf{V}}(y)$ then $\overline{\mathsf{V}}(z)\in\overline{\mathsf{V}}(y)$. For this we fix any $w\in \overline{\mathsf{V}}(y)\setminus \overline{\mathsf{V}}(z)$ and claim that $\overline{\mathsf{V}}(z)\subseteq \overline{\mathsf{V}}(w)$; since $\overline{\mathsf{V}}(w)\in \overline{\mathsf{V}}(y)$, the claim will imply that $\overline{\mathsf{V}}(z)\in\overline{\mathsf{V}}(y)$. To prove the claim we consider any $v\in \overline{\mathsf{V}}(z)$ and show that $v\in\overline{\mathsf{V}}(w)$.  Since $v,w\in \overline{\mathsf{V}}(y)$, by induction assumption for $\overline{\mathsf{V}}(y)$ either $\overline{\mathsf{V}}(v)\in\overline{\mathsf{V}}(w)$, or $\overline{\mathsf{V}}(v)=\overline{\mathsf{V}}(w)$, or $\overline{\mathsf{V}}(w)\in\overline{\mathsf{V}}(v)$. To finish the proof of the claim we just need to rule out the last two cases. Assume for contradiction that $\overline{\mathsf{V}}(v)=\overline{\mathsf{V}}(w)$. Then by \ref{vcl_prop4}. we have $\overline{\mathsf{V}}(w)\in\overline{\mathsf{V}}(z)$. And thus by combination of \ref{vcl_prop1}. and \ref{vcl_prop3}. we should have $w\in\overline{\mathsf{V}}(z)$, contradiction. Now assume for contradiction that $\overline{\mathsf{V}}(w)\in\overline{\mathsf{V}}(v)$.  By combination of \ref{vcl_prop4}. and \ref{vcl_prop2}. we have $\overline{\mathsf{V}}(w)\in\overline{\mathsf{V}}(z)$. Which again leads to a contradiction.

 Let us consider any $y,z\in \overline{\mathsf{V}}(x)$ and show that either $\overline{\mathsf{V}}(y)\in \overline{\mathsf{V}}(z)$, or $\overline{\mathsf{V}}(y)= \overline{\mathsf{V}}(z)$, or $\overline{\mathsf{V}}(z)\in \overline{\mathsf{V}}(y)$. Let $w=\overline{\mathsf{V}}(\overline{\mathsf{V}}(y)\cap \overline{\mathsf{V}}(z))$. Using the fact that we established above we see that either $w= \overline{\mathsf{V}}(y)$ and  $w= \overline{\mathsf{V}}(z)$,  or $w= \overline{\mathsf{V}}(y)$ and $w\in \overline{\mathsf{V}}(z)$, or $w\in \overline{\mathsf{V}}(y)$ and  $w= \overline{\mathsf{V}}(z)$,  or $w\in \overline{\mathsf{V}}(y)$ and $w\in \overline{\mathsf{V}}(z)$. Clearly to finish the proof it is enough to rule out the last case. Assume for a contradiction that $w\in \overline{\mathsf{V}}(y)$ and $w\in \overline{\mathsf{V}}(z)$. We observe that $w\in \overline{\mathsf{V}}(y)\cap \overline{\mathsf{V}}(z)=w$. But by $\varepsilon$-induction it is easy to prove that no set could be its own element.
\end{proof}

Using the established properties of $\overline{\mathsf{V}}$ it is easy to prove the usual set existence axioms: axiom of pair, axiom of union, and the axiom of transitive containment. This allow us to prove in a completely standard fashion that for any two sets $x,y$ there exists the Kuratowski ordered pair $\langle x,y\rangle=\{\{x\},\{x,y\}\}$. For two sets we construct their Cartesian product $X\times Y =\{\langle x,y\rangle \mid x\in X, y\in Y\}$ as a subset of either $\mathcal{P}^2(\overline{\mathsf{V}}(X))$ or $\mathcal{P}^2(\overline{\mathsf{V}}(Y))$.  Thus we could work with binary relations $R\subseteq X\times Y$ and partial functions $f\colon X\to Y$ in a standard fashion.

\subsection{$\mathsf{H}^{\omega}$ and $\mathsf{EA}^{\mathsf{set}}$}

In theory $\mathsf{H}^{\omega}$ we define the class of hereditarily finite sets $\mathsf{HF}$ to be $\bigcup\limits_{n\in \mathsf{Nat}}\overline{\mathsf{V}}(n)$, i.e. that $\mathsf{HF}$ is the union of all the finite levels of von Neumann hierarchy.

\begin{prop} \label{H_and_EA^set}If $\varphi$ is  $\Pi_1^{\mathsf{set}}(\overline{\mathsf{V}})$ sentence, then
$$\mathsf{H}^{\omega}\vdash \varphi^{\mathsf{HF}}\iff \mathsf{EA}^{\mathsf{set}}\vdash \varphi.$$
\end{prop}

In order to prove Proposition \ref{H_and_EA^set} it will be useful to use an alternative axiomatization of $\mathsf{EA}^{\mathsf{set}}$.
\begin{lem}\label{EA_set_alt} The following axioms give an alternative axiomatization of $\mathsf{EA}^{\mathsf{set}}$:
\begin{enumerate}
    \item Extensionality;
    \item $\Delta_0^{\mathsf{set}}(\overline{\mathsf{V}})$-Separation;
    \item Defining axiom for $\mathcal{P}$;
    \item Defining axiom for $\overline{\mathsf{V}}$;
    \item $x=\emptyset\lor \exists y\in x\forall z\in x\; z\not\in y$ (Regularity);
    \item $x\subseteq \overline{\mathsf{V}}(x)\land \overline{\mathsf{V}}(x)\subseteq \mathcal{P}(\overline{\mathsf{V}}(x))$ (operation $\overline{\mathsf{V}}$ maps any $x$ to a transitive set containing $x$);
    \item $x=\emptyset\lor \exists y\;(\overline{\mathsf{V}}(x)=\overline{\mathsf{V}}(\mathcal{P}(y)))$ (every non-empty set lies in a successor level of von Neumann hierarchy).
\end{enumerate}
\end{lem}
\begin{proof} First we verify that the initial axiomatization of $\mathsf{EA}^{\mathsf{set}}$ proves all the axioms of the alternative axiomatization that were not present in original axiomatization. We prove regularity in $\mathsf{EA}^{\mathsf{set}}$ by $\varepsilon$-induction. Using Lemma \ref{vcl} we prove in $\mathsf{EA}^{\mathsf{set}}$ that operation $\overline{\mathsf{V}}$ maps any $x$ to a transitive set containing $x$. And using Lemma \ref{vcl} Claim \ref{vcl_prop6} we prove by adduction induction on $x$ that any $x$ is either empty or $\overline{\mathsf{V}}(x)=\overline{\mathsf{V}}(\mathcal{P}(y))$, for some set $y$.

In other direction we need to verify $\Delta_0^{\mathsf{set}}(\overline{\mathsf{V}},\mathcal{P})$-separation and $\Delta_0^{\mathsf{set}}(\overline{\mathsf{V}},\mathcal{P})$-adduction induction in the alternative axiomatization. 

Now let us prove $\Delta_0^{\mathsf{set}}(\overline{\mathsf{V}},\mathcal{P})$-separation. We reason in the alternative axiomatization of $\mathsf{EA}^{\mathsf{set}}$. Consider a set $a$ and a $\Delta_0^{\mathsf{set}}(\overline{\mathsf{V}},\mathcal{P})$-formula $\varphi(x)$. Our goal is to construct set $\{x\in a\mid \varphi(x)\}$. For this we will find $\Delta_0^{\mathsf{set}}(\overline{\mathsf{V}})$-formula $\psi(x)$ with additional parameters such that $\forall x\in a\;(\varphi(x)\mathrel{\leftrightarrow}\psi(x))$. Let $p_1,\ldots,p_n$ be all the parameters of $\varphi(x)$ and let $k$ be the number of $\mathcal{P}$-symbols used in $\varphi$. We consider sets $b_1=\overline{\mathsf{V}}(\mathcal{P}^{k+1}(x)),\ldots, b_n=\overline{\mathsf{V}}(\mathcal{P}^{k+1}(x))$. Observe that for any $x\in a$ the ranges of all bounded quantifiers within $\varphi(x)$ are covered by some set $b_i$. Moreover, for any value $v$ of a term $t(\vec{y})$ from $\varphi(x)$, under a substitution within the range of bounded quantifiers, we will have that $w\in b_i$ and  $\mathcal{P}(w)\subseteq b_i$, for some $b_i$. Now using $b_i$-bounded quantifiers it is easy to transform $\varphi(x)$ to formula $\psi(x)$ with the desired property.

Now in $\mathsf{EA}^{\mathsf{set}}$ combining regularity, $\Delta_0^{\mathsf{set}}(\mathcal{P},\overline{\mathsf{V}})$-separation, and the fact that any set is contained in a transitive set we easily deduce $\varepsilon$-induction for $\Delta_0^{\mathsf{set}}(\mathcal{P},\overline{\mathsf{V}})$ formulas. Hence the alternative axiomatization of $\mathsf{EA}^{\mathsf{set}}$ contains theory that we have used in the proof of Lemma \ref{vcl}. And therefore all the facts about $\overline{\mathsf{V}}$ from Lemma \ref{vcl} are provable in the alternative axiomatization as well. As in the case of $\mathsf{EA}^{\mathsf{set}}$, the variant of Lemma \ref{vcl} for the alternative axiomatization of $\mathsf{EA}^{\mathsf{set}}$ allows us to construct there ordered pairs, Cartesian products and hence freely work with binary relations and functions.

Let us prove by $\varepsilon$-induction in alternative axiomatization of $\mathsf{EA}^{\mathsf{set}}$ that for any set $x$ if $x=\overline{\mathsf{V}}(x)$, then there exists a linear order $\le_x$ on $x$ such that 
\begin{enumerate}
    \item $\le_x$ extends $\subseteq$ relation, \item for any $y$ we have $y\le_x \{y\}$, if $y,\{y\}\in x$, \item the order $\le_x$ is well-founded, \item the order inverse to $\le_x$ is also well-founded.
\end{enumerate}
The case of empty $x$ is trivial. In the case of non-empty $x$ we fix $y$ such that $\overline{\mathsf{V}}(\mathcal{P}(y))=x$. By Lemma \ref{vcl}, $\mathcal{P}(\overline{\mathsf{V}}(y))=x$. Since $y=\overline{\mathsf{V}}(y)$ and $y\in x$, we could use $\varepsilon$-induction assumption for $y$, i.e. the existence of $\le_y$. We define $\le_x$ to be 
$$z_1\le_x z_2\stackrel{\mbox{\footnotesize def}}{\iff} z_1=z_2\mbox{ or }z_1\ne z_2\land \min_{\le_y}(z_1 \bigtriangleup z_2)\not\in z_1.$$ Essentially, $\le_x$ is the lexicographic order induced by $\le_y$. A routine check shows that $\le_x$ have the desired properties.

We derive $\Delta_0^{\mathsf{set}}(\overline{\mathsf{V}},\mathcal{P})$ adduction induction in the alternative axiomatization as follows. Suppose a $\Delta_0^{\mathsf{set}}(\overline{\mathsf{V}},\mathcal{P})$-property $\varphi(x)$ is adductively progressive (i.e. the premise of adduction induction holds for it). Let us fix $a$ and prove $\varphi(a)$. We consider the set $b=\mathcal{P}(\overline{\mathsf{V}}(a))$ and correspondign order $\le_{b}$. Observe that $\varphi(x)$ is progressive for this order, i.e. $\forall x\in b\;(\forall y<_b x\;\varphi(y)\to\varphi(x))$. And since $\varphi(x)$ is  $\Delta_0^{\mathsf{set}}(\mathcal{P},\overline{\mathsf{V}})$-formula, using $\Delta_0^{\mathsf{set}}(\mathcal{P},\overline{\mathsf{V}})$-separation and well-foundedness of $\le_b$ we show that $\forall x\in b\;\varphi(x)$. In particular we have $\varphi(a)$.\end{proof}

\begin{rem} With additional efforts one could show that even if we remove axiom of regularity from axiomatization in Lemma \ref{EA_set_alt}, the resulting system still will be deductively equivalent to $\mathsf{EA}^{\mathsf{set}}$.  A finite axiomatization of $\mathsf{EA}^{\mathsf{set}}$ could be achieved by replacement of the scheme of $\Delta_0^{\mathsf{set}}$-separation by a version of the axioms of rudimentary closure that accounts for $\overline{\mathsf{V}}$-function (for a more usual version of axioms of rudimentary closure see \cite[Section~VII.2]{Sim09}).
\end{rem}

Now we are ready to prove Proposition \ref{H_and_EA^set}.
\begin{proof}
   Suppose $\varphi$ is of the form $\forall \vec{x}\;\psi(\vec{x})$, where $\psi$ is $\Delta_0^{\mathsf{set}}(\overline{\mathsf{V}})$ formula.  Everywhere in the proof we will use the alternative axiomatization of $\mathsf{EA}^{\mathsf{set}}$ form Lemma \ref{EA_set_alt}.
  
  First assume that $\mathsf{H}^{\omega}\vdash \varphi^{\mathsf{HF}}$. We fix some proof $p$ of $\varphi^{\mathsf{HF}}$ in $\mathsf{H}^{\omega}$. We find a natural number $n$ such that all the sorts of objects used in $p$ have indexes $\le n$. Let us reason in $\mathsf{EA}^{\mathsf{set}}$ and prove  $\varphi$. We consider some sets $\vec{a}$ and claim that $\psi(\vec{a})$. We consider any set $b$ such that all $a$'s are elements of $b$. Next we consider model $\mathfrak{M}$ of the signature of $\mathsf{H}^{\omega}$ restricted to the sorts  $\le n$, where the domain $d_0$ of the sort $0$ is $\overline{\mathsf{V}}(b)$, the domain $d_{i+1}$ of a sort $i+1$ is $\mathcal{P}(d_i)$, and all the membership predicates are interpreted as $\in$. In a straightforward manner we carry out the proof $p$ inside $\mathfrak{M}$: we assemble our $\mathsf{EA}^{\mathsf{set}}$ proof from checks that $\mathfrak{M}\models \chi$ for all the (logical or non-logical) axioms $\chi$ that occur in $p$ and then just follow the inference rules that were used in $p$. This way we ensure that $\mathfrak{M}\models \varphi^{\mathsf{HF}}$. By adduction induction we show that all the sets in $\mathfrak{M}$ are elements of  $\mathsf{HF}^{\mathfrak{M}}$. Hence $\mathfrak{M}\models \varphi$ and thus $\mathfrak{M}\models \psi(\vec{a})$. Observe that since the first-order part of $\mathfrak{M}$ is a transitive model with the standard interpretation of $\overline{\mathsf{V}}$, the $\Delta_0^{\mathsf{set}}(\overline{\mathsf{V}})$ formula $\psi$ is absolute for $\mathfrak{M}$. Hence we conclude that $\psi(\vec{a})$. 
  
  Let us now assume that $\mathsf{H}^{\omega}\nvdash \varphi^{\mathsf{HF}}$ and show that $\mathsf{EA}^{\mathsf{set}}\nvdash \varphi$. We have a model $\mathfrak{M}$ of $\mathsf{H}^{\omega}$, where $\varphi^{\mathsf{HF}}$ fails.  We will construct a model $\mathfrak{M}^u\models \mathsf{EA}^{\mathsf{set}}+\lnot \varphi$. But first we will construct from $\mathfrak{M}$ a model $\mathfrak{M}^f$ of $\mathsf{H}^{\omega}$ that is a counter-model for $\varphi$ but also  $$\mathfrak{M}^f\models \mathsf{V}=\mathsf{HF}\;\;\;\mbox{ and }\;\;\;\mathfrak{M}^f\models\forall x\; (x\subseteq \overline{\mathsf{V}}(s)),\mbox{ for some $s\in \mathsf{HF}^{\mathfrak{M}^f}$.}$$   
  We fix some $\vec{a}=( a_1,\ldots,a_k)$ consisting of elements of $\mathsf{HF}^{\mathfrak{M}}$ such that $\mathfrak{M}\not\models \psi(\vec{a})$. By Lemma \ref{vcl} there should be $i_0$ such that $\mathfrak{M}\models \overline{\mathsf{V}}(a_i)\subseteq \overline{\mathsf{V}}(a_{i_0})$, for all $j$ from $1$ to $k$. We put $s=a_{i_0}$.  We obtain $\mathfrak{M}^{f}$ from $\mathfrak{M}$ by restricting the sets of type $0$ to $\mathfrak{M}$-subsets of $\overline{\mathsf{V}}(s)$. More formally $\mathfrak{M}^{f}$ is a submodel of $\mathfrak{M}$, where we restrict type $0$ to $\mathfrak{M}$-class $D_0$ consisisting of all $\mathfrak{M}$-subsets of $\overline{\mathsf{V}}(s)$. And where we restrict types $n+1$ to $n+2$-sets $D_{n+1}$ in $\mathfrak{M}$ such that $D_{n+1}=(\mathcal{P}(D_n))^{\mathfrak{M}}$.
  
   We define model $\mathfrak{M}^{u}$ to be the collapse of the higher types in  $\mathfrak{M}^{f}$ to an untyped set structure. Formally we define in $\mathfrak{M}^{f}$ the relations $\in_i$ on the sets of the type $i$: $\in_0$ is just $\in$, and $\in_{i+1}$ is 
 $$a^{(i+1)}\in_{i+1}b^{(i+1)}\defiff (\exists x^{(i)}\ain b^{(i+1)}) \forall y^{(i)}(y^{(i)}\ain a^{(i+1)}\mathrel{\leftrightarrow} y^{(i)}\in_i x^{(i)}).$$ 
For each $i$ we define the model $\mathfrak{M}^u_i$ to be the model of pure set-theoretic signature which domain consists of all the sets of the type $i$ from $\mathfrak{M}$ and $\in$ is interpreted as $\in_i$.  Naturally we have end-embeddings of $\mathfrak{M}^u_i$ into $\mathfrak{M}^u_{i+1}$
$$e_i\colon A^{(i)}\longmapsto \{X^{(i)}\mid X^{(i)}\in_i A^{(i)}\}.$$ 
To simplify our notations we will assume (without loss of generality that the sequence $\mathfrak{M}^u_i$ just form a sequence of expanding models (e.g. $\mathfrak{M}^u_i\subseteq \mathfrak{M}^u_{i+1}$ and for any $a\in \mathfrak{M}^u_i$ we have $e_i(a)=a$). We expand $\mathfrak{M}^u_i$ by total functions $\overline{\mathsf{V}}(x)$ and partial functions $\mathcal{P}(x)$. The function $\overline{\mathsf{V}}$ in $\mathfrak{M}^u_0$ coincides with $\overline{\mathsf{V}}$ from $\mathfrak{M}^f$.  The function $\overline{\mathsf{V}}$ in $\mathfrak{M}^u_{i+1}$ extends $\overline{\mathsf{V}}$ from $\mathfrak{M}^u_i$ by mapping any $a\in \mathfrak{M}^u_{i+1}\setminus  \mathfrak{M}^u_i$ to the $\subseteq$-greatest element of $\mathfrak{M}^u_{i+1}$ (in $\mathfrak{M}^f$ it is the type $i+1$ set consisting of all type $i$ sets).  We define $\mathcal{P}$ in $\mathfrak{M}^u_{i+1}$ for all $a\in \mathfrak{M}^u_{i}$ to be the powerset of $a$ (inside $\mathfrak{M}^u_{i+1}$). It is easy to see that $\mathcal{P}$ from $\mathfrak{M}^u_{i}$ extends $\mathcal{P}$ from $\mathfrak{M}^u_{j}$, for $i>j$. And that for any $a\in \mathfrak{M}^u_i$, the powerset $\mathcal{P}(a)$ is defined in $\mathfrak{M}^u_{i+1}$. The model $\mathfrak{M}^u$ is the union of all $\mathfrak{M}^u_i$'s.  Clearly $\mathfrak{M}^u$ is a model with total powerset function.

A routine check shows that all the models $\mathfrak{M}^u_i$ satisfy all the axioms of the alternative axiomatizations of $\mathsf{EA}^{\mathsf{set}}$ other than the defining axiom for $\mathcal{P}$. Since all this axioms were $\Pi_2^{\mathsf{set}}(\overline{\mathsf{V}})$-sentences, they also holds in  $\mathfrak{M}^u$. Hence $\mathfrak{M}^u$  is a model of $\mathsf{EA}^{\mathsf{set}}$.
\end{proof}

\begin{rem} By the same technique as above it is easy to prove that for $\Pi_1^{\mathsf{set}}$ sentences
$$\mathsf{wEA}^{\mathsf{set}}\vdash \varphi\iff \mathsf{H}^{\omega}\vdash \varphi,$$
where $\mathsf{wEA}^{\mathsf{set}}$ is $\mathsf{EA}^{\mathsf{set}}$ with the scheme of adduction induction for $\Delta_0(\mathcal{P},\overline{\mathsf{V}})$-formulas replaced by the scheme of $\varepsilon$-induction for $\Delta_0(\mathcal{P},\overline{\mathsf{V}})$-formulas.
\end{rem}

\subsection{Cardinal arithmetic in $\mathsf{EA}^{\mathsf{set}}$ and the theory $\mathsf{EA}$}
The standard informal definition of cardinal numbers is that they are equivalence classes of sets with respect to equinumerocity relation. Unfortunately this equivalence classes do not form sets. And in order to work with cardinal numbers as individual sets they should be represented as some sets from which it is possible to recover the respective equivalence class. The most well-known solution to this (that is typically used in $\mathsf{ZFC}$) is to define the cardinal number $|x|$ to be the least ordinal $\alpha$ such that there is a bijection $f\colon x\to \alpha$.  A different solution (that sometimes is used in $\mathsf{ZF}$) is to define cardinal number $|x|$ to be the set $\{y\mid y\in \mathsf{V}_{\alpha}\mbox{ and there is a bijection $f\colon x\to y$}\}$, where $\alpha$ is the least ordinal for which there exists $y\in \mathsf{V}_{\alpha}$ and a bijection  $f\colon x\to y$. It is possible to show that $\mathsf{EA}^{\mathsf{set}}$ does prove axiom of choice and Zermelo theorem. However, $\mathsf{EA}^{\mathsf{set}}$ does not prove Mostowski transitive collapse theorem and even that every set is equinumerous to an ordinal. Thus the $\mathsf{ZFC}$-style definition of cardinal numbers is not suitable for $\mathsf{EA}^{\mathsf{set}}$.

We will use $\mathsf{ZF}$-style cardinals within $\mathsf{EA}^{\mathsf{set}}$. We put $$|a|=\{y\mid y\in \overline{\mathsf{V}}(b)\mbox{ and there is a bijection $f\colon a\to y$}\},$$ where $\overline{\mathsf{V}}(b)$ is the smallest level of von-Neumann hierarchy such that there exists at least one $y\in \overline{\mathsf{V}}(b)$ equinumerous with $a$.
For two cardinals $c,d$ we write $c\le d$ if there is an injection from some $x\in c$ into some $y\in d$.  For two cardinals $c,d$ the cardinal $c+d$ is the cardinality of the disjoint union $|x\sqcup y|$, for some $x\in c$ and $y\in d$ (as usual $x\sqcup y= x\times \{\emptyset\}\cup y\times \{\{\emptyset\}\}$).  For two cardinals $c,d$ the cardinal $cd$  is the cardinal of Cartesian product $|x\times y|$,  for some $x\in c$ and $y\in d$. For a cardinal $c$ the cardinal $2^c$  is the cardinal of the powerset $|\mathcal{P}(x)|$,  for some $x\in c$. It is easy to check that the definitions of addition, multiplication, and exponentiation indeed give well-defined functions.

Henceforth the cardinal arithmetic gives us an embedding $\mathcal{CRD}$ of the arithmetical language with the predicates $=,\le$, constant  $0$, and functions $S,+,\times,x\longmapsto 2^x$ into the theory $\mathsf{EA}^{\mathsf{set}}$.  Below we will show that $\mathcal{CRD}$ is an interpretation of Kalmar elementary functions arithmetic $\mathsf{EA}$.

Recall that $\mathsf{EA}$ is a first-order theory, which language is the arithmetical language with exponentiation function (as in the paragraph above). The non-logical axioms of $\mathsf{EA}$ are
\begin{enumerate}
\item \label{EA_u1} $S(x)\ne 0$;
\item $S(x)=S(y)\to x=y$;
\item $x+0=x$;
\item $x+S(y)=S(x+y)$;
\item $x0=0$;
\item $xS(y)=xy+x$;
\item $2^0=1$;
\item $2^{S(x)}=2^x+2^x$;
\item $x\le 0\mathrel{\leftrightarrow} x=0$;
\item \label{EA_u10} $x\le S(y) \mathrel{\leftrightarrow} x\le y\lor x=S(y)$;
\item \label{EA_i} $ \varphi(0)\land \forall x\; (\varphi(x)\to \varphi(S(x)))\to \forall x\;\varphi(x)$, where $\varphi$ is $\Delta_0$ ($\Delta_0\mbox{-}\mathsf{Ind}$).
\end{enumerate}

\begin{prop}\label{CRD_int} $\mathcal{CRD}$ is an interpretation of $\mathsf{EA}$ in $\mathsf{EA}^{\mathsf{set}}$.\end{prop}
\begin{proof}
In a straightforward manner we prove the $\mathcal{CRD}$-translations of the axioms \ref{EA_u1}.--\ref{EA_u10}. of $\mathsf{EA}$.

Observe that the $\mathcal{CRD}$-translation of an instance of mathematical induction
$$\mathcal{CRD}( \varphi(0)\land \forall x\; (\varphi(x)\to \varphi(S(x)))\to \forall x\;\varphi(x))$$
is implied by the following instance of adduction induction
$$\forall x,y (\mathcal{CRD}(\varphi(|x|))\land \mathcal{CRD}(\varphi(|y|))\to \mathcal{CRD}(\varphi(|x\cup\{y\}|)))\to \forall x\; \mathcal{CRD}(\varphi(|x|)).$$
Thus in order to prove that $\Delta_0\mbox{-}\mathsf{Ind}$ holds in the interpretation $\mathcal{CRD}$ it is enough to show that for any  $\Delta_0$  formula $\varphi(x,y_1,\ldots,y_n)$ the formula $\mathcal{CRD}(\varphi(|x|,|y_1|,\ldots,|y_n|))$ is equivalent to a $\Delta_0^{\mathsf{set}}(\mathcal{P},\overline{\mathsf{V}})$ formula in $\mathsf{EA}^{\mathsf{set}}$.

Suppose $f(x_1,\ldots,x_n)$ is a definable over $\mathsf{EA}^{\mathsf{set}}$ function, i.e. there is a fixed formula $\mathsf{Gr}_{f}(x_1,\ldots,x_n,y)$ defining the graph of $f$ such that $$\mathsf{EA}^{\mathsf{set}}\vdash\forall x_1,\ldots,x_n\exists! y\;(\mathsf{Gr}_f(x_1,\ldots,x_n,y)).$$ We say that $f(x_1,\ldots,x_n)$ have constant rank property if additionally over $\mathsf{EA}^{\mathsf{set}}$ the formula $\mathsf{Gr}_f$ is equivalent to a $\Delta_0^{\mathsf{set}}(\mathcal{P},\overline{\mathsf{V}})$ formula and there is a number $r_f$ such that $$\begin{aligned} \mathsf{EA}^{\mathsf{set}}\vdash (\forall& x_1,\ldots,x_n,p,y)\;(\mathsf{Gr}_f(x_1,\ldots,x_n,y)\to \\ & y\in \mathcal{P}^{r_f+1}(\overline{\mathsf{V}}(x_1))\lor\ldots \lor y\in \mathcal{P}^{r_f+1}(\overline{\mathsf{V}}(x_n))\lor y\in \mathcal{P}^{r_f+1}(\overline{\mathsf{V}}(p))).\end{aligned}$$
Obviously, $\mathcal{P}$ and $\overline{\mathsf{V}}$ have constant rank property with $r_{\mathcal{P}}=1$ and $r_{\overline{\mathsf{V}}}=0$.

Let functions $f_1,\ldots,f_n$ have constant rank property. And let us consider the naturally defined class of formulas $\Delta_0^{\mathsf{set}}(\mathcal{P},\overline{\mathsf{V}},f_1,\ldots,f_n)$ that consists of all the formulas with bounded quantifiers in the signature with predicates $\in,=$ and the functions $\mathcal{P},\overline{\mathsf{V}},f_1,\ldots,f_n$. We prove that each $\Delta_0^{\mathsf{set}}(\mathcal{P},\overline{\mathsf{V}},f_1,\ldots,f_n)$ formula $\varphi(\vec{x})$ is $\mathsf{EA}^{\mathsf{set}}$-provably equivalent to a $\Delta_0^{\mathsf{set}}(\mathcal{P},\overline{\mathsf{V}})$ formula $\varphi'(\vec{x},p)$. We achieve this in two steps. First, by induction on a term construction we show that each term $t(\vec{x})$ built of functions with constant rank property is itself a function with constant rank property. For induction step we use Lemma \ref{vcl} Claim \ref{vcl_prop6} to show that if $f(x_1,\ldots,x_n)$ and $u_1(y_1,\ldots,y_m), \ldots,u_n(y_1,\ldots,y_m)$ have constant rank properties, then $g(y_1,\ldots,y_m)=f(u_1(y_1,\ldots,y_m), \ldots,u_n(y_1,\ldots,y_m))$ have constant rank property with $r_g=r_f+\max(r_{u_1},\ldots,r_{u_n})$. Second, by induction on construction of $\Delta_0^{\mathsf{set}}(\mathcal{P},\overline{\mathsf{V}},f_1,\ldots,f_n)$ formulas $\psi(\vec{x})$ we construct $\mathsf{EA}^{\mathsf{set}}$-provably equivalent $\Delta_0^{\mathsf{set}}(\mathcal{P},\overline{\mathsf{V}})$ formulas $\psi'(\vec{x},p)$. Here both the base and the step of induction are easy to justify using the fact that any term built of $\mathcal{P},\mathcal{V},f_1,\ldots,f_n$ is a function with constant rank property.

It is easy to see that all the functions of cardinal arithmetic (including constant $0$) have constant rank property. That the comparison of cardinals is definable by a $\Delta_0^{\mathsf{set}}(\mathcal{P},\overline{\mathsf{V}})$ formula. And that the functions $x\longmapsto |x|$, $\mathsf{dc}\colon x\longmapsto \{|y|\mid |y|\le |x|\}$ have constant rank property. 

In terms of $\mathsf{dc}$ we easily replace cardinality-bounded quantifiers with bounded membership quantifiers: formula $\forall y(|y|\le |t| \to \varphi(|y|))$ is $\mathsf{EA}^{\mathsf{set}}$-equivalent to $\forall y\in \mathsf{dc}(t)\;\varphi(y)$. Thus, for any arithmetical $\Delta_0$-formula $\psi(x_1,\ldots,x_n)$, the translation $\mathcal{CRD}(\psi(|x_1|,\ldots,|x_n|))$ is $\mathsf{EA}^{\mathsf{set}}$-provably equivalent to some $\Delta_0^{\mathsf{set}}(\mathcal{P},\overline{\mathsf{V}},0,S,+,\times,x\longmapsto 2^x,x\longmapsto |x|,\mathsf{dc})$ formula $\psi'(x_1,\ldots,x_n)$. And by the above $\psi'$ is $\mathsf{EA}^{\mathsf{set}}$-provably equivalent to a $\Delta_0^{\mathsf{set}}(\mathcal{P},\overline{\mathsf{V}})$-formula $\psi''(x_1,\ldots,x_n,p)$.

Applying this construction to the case of formula $\varphi(x)$ we conclude the proof of the lemma.
\end{proof}



\subsection{Bi-interpretability of $\mathsf{EA}^{\mathsf{set}}$ and $\mathsf{EA}$}\label{bi-interpretabilit_section}
In this section we show that, this two theories enjoy nicer connection than just existence of an interpretation of $\mathsf{EA}$ in $\mathsf{EA}^{\mathsf{set}}$. Namely, we will show that the interpretation $\mathcal{CRD}$ together with Ackermann's interpretation of $\mathsf{EA}^{\mathsf{set}}$ in $\mathsf{EA}$ form a bi-interpretation between the theories.

The notion of bi-interpretability is a strong equivalence between first-order theories. For example, there is a bi-interpretation between Tarski's geometry and the theory of real closed fields. The notion most naturally could be defined in the category-theoretic setting: two theories are called bi-interpretable if there is an equivalence between them in $2$-category of interpretations (see \cite{Vis06}). In more explicit terms, a bi-interpretation between first-order theories $T$ and $U$ is a tuple $\langle i,j,k,l\rangle$ such that
\begin{enumerate}
    \item $i\colon T\triangleleft U$ is an interpretation of $T$ in $U$,
    \item $j\colon U\triangleleft T$ is an interpretation of $U$ in $T$,
    \item  $k$ is a definable isomorphism between the interpretation $\mathsf{id}_U\colon U\triangleright U$ and the composition $j\circ i$,
    \item  $l$ is a definiable isomorphism between the interpretation $\mathsf{id}_T\colon T\triangleright T$ and $i\circ j$.
\end{enumerate}
A definable isomorphism $u$ between iterpretations $i_1,i_2\colon T\triangleleft U$ is an $U$-definable bijection between the domains of $i_1$ and $i_2$ such that in any model $\mathfrak{M}$ of $U$ the function $u^{\mathfrak{M}}$ is an isomorphism of the models $i_1^{\mathfrak{M}}$ and $i_2^{\mathfrak{M}}$ of the theory $T$.


Recall that the Ackermann's membership predicate $ n\in_{\mathsf{Ack}} m$ is ``the $n$-th bit of the number $m$ is equal to $1$''. It could be naturally defined in the arithmetical language:
$$x\in_{\mathsf{Ack}}y \defiff (\exists z, w)  \;(y=z2^{x+1}+w\land w<2^{x+1}\land w\ge 2^x).$$

Theory $\mathsf{EA}$ could develop number of standard set-theoretic constructions in term of Ackermann's membership (see the book by H{\'a}jek and Pudl{\'a}k \cite[Section~I.1(b)]{HajPud17}). In particular in \cite{HajPud17} it have been proved in $\mathsf{EA}$ that $\in_{\mathsf{Ack}}$ satisfies extensionality and powerset axioms. The latter fact allows us to define in $\mathsf{EA}$ the function $\mathcal{P}_{\mathsf{Ack}}(x)$ that maps a number $x$ to its powerset with respect to $\in_{\mathsf{Ack}}$. 

Moreover, we naturally could define $\overline{\mathsf{V}}_{\mathsf{Ack}}$ function. It is a well-known fact that it is possible to define the graph of the superexponentiation function 
\begin{equation}\label{supexp_def}2^x_0=x,\;\;\;\;2^{x}_{y+1}=2^{2^x_y}\end{equation}
by a $\Delta_0$ arithmetical formula and to prove in $\mathsf{EA}$ that the partial function $2^x_y$ satisfies the equalities (\ref{supexp_def}) (in the sense that both the sides of the equalities are simultaneously defined or undefined and if they are defined then they are equal).  We define $\overline{\mathsf{V}}_{\mathsf{Ack}}(x)$ to be the least $y$ of the form $2_z^{0}-1$  such that $y\ge x$. It is easy to see that  $\mathsf{EA}$ proves that $\overline{\mathsf{V}}_{\mathsf{Ack}}(x)$ is a total function and that $\overline{\mathsf{V}}_{\mathsf{Ack}}(x)\le 2^x$. Simple check shows $\mathsf{EA}$ proves the defining axiom for $\overline{\mathsf{V}}$ within this interpretation. 

Thus we have defined embedding $\mathcal{ACK}$ of the language of $\mathsf{EA}^{\mathsf{set}}$ into the language of arithmetic. Clearly, the predicate $\in_{\mathsf{Ack}}$ and the functions $\overline{\mathsf{V}}(x)$, $\mathcal{P}(x)$ are Kalmar elementary. Therefore,  for each set-theoretic $\Delta_0^{\mathsf{set}}(\mathcal{P},\overline{\mathsf{V}})$ formula $\varphi(\vec{x})$ we could find $\Delta_0$ formula $\varphi'(\vec{x})$ that is $\mathsf{EA}$-provably equivalent to $\mathcal{ACK}(\varphi(\vec{x}))$. This allows us to prove in $\mathsf{EA}$ the $\mathcal{ACK}$-translations of all the instances of $\Delta_0^{\mathsf{set}}(\mathcal{P},\overline{\mathsf{V}})$-separation and $\Delta_0^{\mathsf{set}}(\mathcal{P},\overline{\mathsf{V}})$ adduction induction. This concludes the proof of the fact that $\mathcal{ACK}$ is an interpretation of $\mathsf{EA}^{\mathsf{set}}$ in $\mathsf{EA}$.

The isomorphism $k$ between $id_{\mathsf{EA}}$ and $\mathcal{CRD}\circ \mathcal{ACK}$ is the Kalmar elementary function that maps a number $x$ to the $\in_{\mathsf{Ack}}$-cardinal that represent the class of equivalence of $\in_{\mathsf{Ack}}$-sets with precisely $x$ elements. 

The isomorphism $l$ between $id_{\mathsf{EA}^{\mathsf{set}}}$ and $\mathcal{ACK}\circ \mathcal{CRD}$ should be a function that maps a set $x$ to the cardinal number $c$  that represents the set $x$ with respect to $\mathcal{CRD}(\in_{\mathsf{Ack}})$. We see that $l$ should satisfy the following equation \begin{equation}\label{l_property}l(x)=\sum\limits_{y\in x} 2^{l(y)}.\end{equation}
Here the cardinal $\sum\limits_{y\in x} 2^{l(y)}$ could be formally defined as the cardinality of $\bigcup\limits_{y\in x} f(y)\times \{y\}$, where $f$ is any set-size function with $\mathsf{dom}(f)=x$ that maps a set $y$ to a set with cardinality $2^{l(y)}$.
We prove by $\varepsilon$-induction on sets $a$ that partial function $l_a\colon \overline{\mathsf{V}}(a)\to \mathcal{P}(\overline{\mathsf{V}}(a))$ that satisfies the equation (\ref{l_property}) exists and unique. And then we define the desired function $l$ to be the union of all $l_a$'s.  Using (\ref{l_property}) it is straightforward to show that we have indeed defined the desired isomorphism $l$.

Thus we have proved
\begin{thm}\label{bi-interpretability_theorem}
   The interpretation $\mathcal{ACK}\colon \mathsf{EA}\triangleright \mathsf{EA}^{\mathsf{set}}$ and the interpretation  $\mathcal{CRD}\colon \mathsf{EA}^{\mathsf{set}}\triangleright \mathsf{EA}$ form a bi-interpretation. Hence $\mathsf{EA}$ proves a sentence $\varphi$ whenever $\mathsf{EA}^{\mathsf{set}}$ proves $\mathcal{CRD}(\varphi)$ and $\mathsf{EA}^{\mathsf{set}}$ proves a sentence $\psi$ whenever $\mathsf{EA}^{\mathsf{set}}$ proves $\mathcal{ACK}(\psi)$ 
\end{thm}

\begin{cor}\label{HCST_EA_ON_cor}
  For any $\Pi_1^{\mathsf{set}}(\overline{\mathsf{V}})$ sentence  $\varphi$
  $$\mathsf{H}^{\omega}\vdash \varphi^{\mathsf{HF}}\iff \mathsf{EA}\vdash \mathcal{ACK}(\varphi).$$
\end{cor}

\subsection{Ordinal arithmetic and superexponential cut}
 As in the case of $\mathsf{H}^{\omega}$ we define ordinals in $\mathsf{EA}^{\mathsf{set}}$ to be transitive sets consisting only of transitive sets. Clearly, we could express by a $\Delta_0^{\mathsf{set}}(\mathcal{P},\overline{\mathsf{V}})$ formula $x\in \mathsf{On}$ the fact that set $x$ is an ordinal. 
 
 Using powerset axiom it is easy to prove the totality of the successor function on ordinals. We want to give $\Delta_0^{\mathsf{set}}(\mathcal{P},\overline{\mathsf{V}})$ formulas for graphs of partial functions $+,\times,\alpha\longmapsto 2^\alpha$ of ordinal arithmetic that satisfy the standard recursive definitions:
\begin{enumerate}
\item $\alpha+\beta=\sup(\{\alpha\}\cup\{S(\alpha+\gamma)\mid \gamma<\beta\})$;
\item $\alpha\cdot\beta=\sup\{\alpha\cdot \gamma+\alpha\mid \gamma<\beta\}$;
\item $2^\alpha=\sup\{2^\beta+2^\beta\mid \beta<\alpha\}$.
\end{enumerate}
Since all the functions are defined in the same manner, we give the definition only for addition function. We consider the partial addition functions $+_\delta\colon \delta\times\delta\to \delta$ such that
\begin{enumerate}
\item $\alpha+_{\delta}\beta$ is defined iff
\begin{enumerate}
\item $\alpha,\beta\in \delta$,
\item for all $\gamma<\beta$ the value $\alpha+_\delta\gamma$ is defined, 
\item $\sup(\{\alpha\}\cup\{S(\alpha+_\delta\gamma)\mid \gamma<\beta\})\in \delta$;
\end{enumerate}
\item if $\alpha+_\delta\beta$ is defined then $\alpha+_\delta\beta=\sup(\{\alpha\}\cup\{S(\alpha+_\delta\gamma)\mid \gamma<\beta\})$.
\end{enumerate}
The existence and uniqueness of partial functions $+_{\delta}$ is proved by induction on $\delta$. We define $$\alpha+\beta=\gamma\defiff  \alpha+_{S(\gamma)}\beta=\gamma.$$
The ordinal arithmetic gives us an embedding $\mathcal{ON}$ of the predicate-only version of arithmetical language into the language of $\mathsf{EA}^{\mathsf{set}}$.

The proof of Proposition \ref{H_and_EA^set} could be modified to obtain the following:
\begin{lem}\label{HCST_EA^set_ON}
  For any $\Pi_1^{\mathsf{pred}}$ sentence $\varphi$
  $$\mathsf{H}^{\omega}\vdash \mathcal{NAT}(\varphi)\iff \mathsf{EA}^{\mathsf{set}}\vdash \mathcal{ON}(\varphi).$$
\end{lem}
\begin{proof} The proof of Proposition \ref{H_and_EA^set} consisted of two parts: 1. to transform an  $\mathsf{H}^{\omega}$-proof of $\varphi^{\mathsf{HF}}$ into an $\mathsf{EA}^{\mathsf{set}}$-proof of $\varphi$ and 2. to transform an $\mathsf{H}^{\omega}$-model of $\lnot \varphi^{\mathsf{HF}}$ into an $\mathsf{EA}^{\mathsf{set}}$-model of $\lnot \varphi$. 

The analogue of part 1. for the present lemma is a transformation of $\mathsf{H}^{\omega}$-proof of $\mathcal{NAT}(\varphi)$ into an $\mathsf{EA}^{\mathsf{set}}$-proof of $\mathcal{ON}(\varphi)$. The addition to the proof from Proposition \ref{H_and_EA^set} is that we need to verify that the ordinal arithmetic in the model $\mathfrak{M}$ that we obtain from the definition of ordinal arithmetic for the theory $\mathsf{H}^{\omega}$ coincide with the restriction of $\mathsf{EA}^{\mathsf{set}}$ ordinal arithmetic to the model. This could be achieved by a trivial proof by induction (inside $\mathsf{EA}^{\mathsf{set}}$).

And the analogue of part 2. for the present lemma is a transformation of an $\mathsf{H}^{\omega}$-model of $\lnot \mathcal{NAT}(\varphi)$ into an $\mathsf{EA}^{\mathsf{set}}$-model of $\lnot\mathcal{ON}(\varphi)$. Here the modification of the construction from Proposition \ref{H_and_EA^set} is that we need to ensure that the ordinal arithmetic is preserved when we transit from $\mathfrak{M}$ to $\mathfrak{M}^f$ and when we transit from $\mathfrak{M}^f$ to $\mathfrak{M}^u$. Which again could be done by a straightforward arguments by $\varepsilon$-induction (in $\mathsf{H}^{\omega}$).\end{proof}

The superexponential cut $\mathcal{S}$ in $\mathsf{EA}$ is:
$$ x\in\mathcal{S}\defiff 2_x^0 \mbox{ is defined}.$$
It is easy to observe that $\mathsf{EA}$ proves that  $\mathcal{S}$ is a cut, i.e. that $$x\in\mathcal{S} \Rightarrow x+1\in \mathcal{S} \mbox{ and } \forall y\le x (y\in \mathcal{S}).$$
For a sentence $\varphi$ of predicate-only version of arithmetical language, the sentence $\varphi^{\mathcal{S}}$ is the relativization of $\varphi$ to $\mathcal{S}$, i.e.  $\varphi^{\mathcal{S}}$  is $\varphi$, where all the quantifiers $\forall x$ and $\exists x$ are replaced with the quantifiers $(\forall x\in \mathcal{S})$ and $(\exists x\in \mathcal{S})$, respectively.

\begin{lem}\label{EA_EA^set_ON}Suppose $\varphi$ is a sentence of predicate-only version of arithmetical language. Then 
$$\mathsf{EA}\vdash \varphi^{\mathcal{S}}\iff \mathsf{EA}^{\mathsf{set}}\vdash \mathcal{ON}(\varphi).$$\end{lem}
\begin{proof}
First let us prove in $\mathsf{EA}^{\mathsf{set}}$ that the definable function $x\longmapsto |x|$ maps ordinals to the elements of the class $\mathcal{S}^{\mathcal{CRD}}$ (the cut $\mathcal{S}$ inside interpretation $\mathcal{CRD}$). We show this by proving by induction on ordinals $\alpha$ that $\mathcal{CRD}(2^0_x=y)$, where $x=|\alpha|$ and $y=|\overline{\mathsf{V}}(\alpha)|$.  

On the other hand, inside $\mathsf{EA}$ for numbers $x\in \mathcal{S}$ we could define function $\mathsf{on}(x)$: 
$$\mathsf{on}(0)=0;\;\;\;\; \mathsf{on}(x+1)=\mathsf{on}(x)+2^{\mathsf{on}(x)}.$$
It is well-defined since $\mathsf{on}(x)<2^0_{x+1}$. The intuition behind $\mathsf{on}$ is that it maps a number $x$ to the von Neumann ordinal $x$ with respect to $\in_{\mathsf{Ack}}$.

We observe that $x\longmapsto|x|$ and $\mathsf{on}$ are inverse to each other in the following sense. Recall that $k$ is a $\mathsf{EA}$-definable isomorphism between the interpretation $id_{\mathsf{EA}}$ and $\mathcal{ACK}\circ\mathcal{CRD}$ and $l$ is a $\mathsf{EA}$-definable isomorphism between the interpretation $id_{\mathsf{EA}^{\mathsf{set}}}$ and $\mathcal{CRD}\circ\mathcal{ACK}$; both $k$ and $l$ were defined in Section \ref{bi-interpretabilit_section}.  From one side, in $\mathsf{EA}$ we could prove by induction on  $x\in \mathcal{S}$ that $k(|\mathsf{on}(x)|_{\mathsf{Ack}})=x$, where $y\longmapsto|y|_{\mathsf{Ack}}$ is the cardinality function according to the interpretation $\mathcal{ACK}$. From the other side, in $\mathsf{EA}^{\mathsf{set}}$ by induction on ordinals $\alpha$ we show that $l(\mathsf{on}_{\mathsf{Crd}}(|\alpha|))=\alpha$, where $\mathsf{on}_{\mathsf{Crd}}$ is the $\mathsf{on}$ function inside the interpretation $\mathcal{CRD}$.

Now we conclude that the theory $\mathsf{EA}$ proves that $\mathsf{on}$ is a bijection between $\mathcal{S}$ and ordinals according to $\mathcal{ACK}$ interpretation and that $\mathsf{EA}^{\mathsf{set}}$ proves that $x\longmapsto |x|$ is a bijection between the class of ordinals $\mathsf{ON}$ and the cut $\mathcal{S}^{\mathcal{CRD}}$.

By induction we prove in $\mathsf{EA}^{\mathsf{set}}$ that $x\longmapsto |x|$ is an isomorphism between predicate-only arithmetic on ordinals and the predicate-only arithmetic on elements of $\mathcal{S}$ inside $\mathcal{CRD}$. By induction we prove in $\mathsf{EA}$ that $\mathsf{on}$ is an isomorphism between predicate-only arithmetic on $\mathcal{S}$ and the predicate-only arithmetic on ordinals according to the interpretation $\mathcal{ACK}$. 

The last fact about isomorphism together with Theorem \ref{bi-interpretability_theorem} concludes the proof\end{proof}

Combining Lemma \ref{HCST_EA^set_ON} and Lemma \ref{EA_EA^set_ON} we get
\begin{lem}\label{HCST_EA_ON}
  For each $\Pi_1^{\mathsf{pred}}$ sentence   $\varphi$
  $$\mathsf{H}^{\omega}\vdash \mathcal{NAT}(\varphi)\iff \mathsf{EA}\vdash \varphi^{\mathcal{S}}.$$
\end{lem}

\section{Consistency Proof} \label{Non-god_section} 
We fix some natural arithmetization of many sorted first-order logic in arithmetic. Typical arithmetizations of logic that one could find in the literature (for example  \cite{HajPud17}) are arithmetizations of one sorted first-order logic, or just the first-order arithmetical language. However, there are no essential differences between the arithmetizations of one sorted and many sorted logics, thus we will not develop the details of this kind of arithmetization in the present paper.

As it will bee seen from the proofs, our results about provability of consistency are rather robust with respect to the choice of particular formula expressing the fact that something is a proof of contradiction. We need the theory $\mathsf{EA}$ to be able to naturally work with the formulas and proofs. In addition $\mathsf{EA}$ should  be able to prove that for each proof $p$ and formula $\varphi$  (we identify formulas and proofs with their G\"odel numbers) $p\ge 2|\varphi|$, where $|\varphi|$ is the length of $\varphi$. And $\mathsf{EA}$ should prove that for each proof $p$ we have $p\ge 2t_p$, where $t_p$ is the number of distinct types of objects used in $p$. We note that those are very mild assumptions. For example, this conditions are verified if $\mathsf{EA}$ proves that $\varphi\ge 2|\varphi|$, for each formula $\varphi$ and that $p\ge \varphi$, for each proof $p$ and formula $\varphi$ in it. 

Recall that we treat $\mathsf{H}^{\omega}_{<\omega}$ as a many sorted first-order theory. It is fairly obvious that the set of axioms of $\mathsf{H}^{\omega}_{<\omega}$ is Kalmar elementary, i.e. that it is definable by an arithmetical $\Delta_0$ formula. We have $\Sigma_1$ formula $\mathsf{PrfCnt}_{\mathsf{H}^{\omega}_{<\omega}}(p)$ that is a formalization of
\begin{center} \it $p$ is a G{\"o}del number of a Hilbert-style proof of $\exists x\;\lnot x=x$ from axioms of $T$.\end{center}
The formula $\mathsf{Con}(\mathsf{H}^{\omega}_{<\omega})$ is $\forall p\;\lnot \mathsf{PrfCnt}_{\mathsf{H}^{\omega}_{<\omega}}(p)$.

We would like to prove consistency of $\mathsf{H}^{\omega}_{<\omega}$ in $\mathsf{H}^{\omega}$. But we have only the embedding of predicate-only arithmetic language in $\mathsf{H}^{\omega}$ rather than the full arithmetic language. Thus we $\mathsf{EA}$-equivalently transform formula $\mathsf{PrfCnt}_{\mathsf{H}^{\omega}_{<\omega}}(p)$ to a $\Sigma_1^{\mathsf{pred}}$  formula $\mathsf{PrfCnt}^{\mathsf{pred}}_{\mathsf{H}^{\omega}_{<\omega}}(p)$ and $\mathsf{EA}$-equivalently transform the formula $\mathsf{Con}(\mathsf{H}^{\omega}_{<\omega})$ to the form $\mathsf{Con}^{\mathsf{pred}}(\mathsf{H}^{\omega}_{<\omega})$: $$\forall p\; \lnot \mathsf{PrfCnt}^{\mathsf{pred}}_{\mathsf{H}^{\omega}_{<\omega}}(p).$$

\begin{lem} \label{Con_in_EA} $\mathsf{EA}$ proves $(\mathsf{Con}^{\mathsf{pred}}(\mathsf{H}^{\omega}_{<\omega}))^{\mathcal{S}}$.\end{lem}
\begin{proof}Let us reason in $\mathsf{EA}$. Clearly, it would be enough to show that $\lnot\mathsf{PrfCnt}_{\mathsf{H}^{\omega}_{<\omega}}(p)$, for all $p\in \mathcal{S}$. 

For a contradiction we assume that there exists a proof $p\in\mathcal{S}$ of $\exists x\;\lnot x=x$ from axioms of $\mathsf{H}^{\omega}_{<\omega}$. Let $n$ be the greatest number such that the proof $p$ uses the axiom $\exists x\; \mathsf{Nmb}_n(x)$. And suppose that the types used in $p$ are $k_0,\ldots,k_{m-1}$. From the conditions on G\"odel numbering that we have outlined above we see that $p\ge 2n$ and $p\ge 2m$. And since $p\in \mathcal{S}$, the value $2^0_{n+m}$ is defined.

Let us now define a finite model $\mathfrak{M}$ of all the axioms used in $p$. The domain of $k_0$ is the $n$-th level of von Neumann hierarchy $\mathsf{V}_{n+1}$ with respect to $\in_{\mathsf{Ack}}$ (it is the set of all numbers $\le 2^0_{n+1}-1$). The domain of the sort $k_i$ is $\mathsf{V}_{n+i}$ with respect to $\in_{\mathsf{Ack}}$. We interpret $\in$ and all relevant $\ain_i$ predicates as $\in_{\mathsf{Ack}}$, we interpret $\overline{\mathsf{V}}$ by $\overline{\mathsf{V}}_{\mathsf{Ack}}$ from Section \ref{bi-interpretabilit_section}. A straightforward check shows that the defined structure $\mathfrak{M}$ indeed satisfies all the axioms used in $p$. And next we show by induction on subproofs of $p$ that all the formulas in $p$ are satisfied in $\mathfrak{M}$. We note that the latter is possible since the property of a formula to be true in $\mathfrak{M}$ is $\Delta_0$ (this is due to the fact that our version of arithmetic language contains exponentiation). Thus $\exists x\;\lnot x=x$ is satisfied in $\mathfrak{M}$, contradiction.\end{proof}

\begin{thm}\label{H^omega_nongodelian}$\mathsf{H}^{\omega}$ proves $\mathcal{NAT}(\mathsf{Con}^{\mathsf{pred}}(\mathsf{H}^{\omega}_{<\omega}))$.\end{thm}
\begin{proof} Since $\mathsf{Con}^{\mathsf{pred}}(\mathsf{H}^{\omega}_{<\omega})$ is equivalent to $\Pi_1^{\mathsf{pred}}$ sentence, we get the theorem by combination of Lemma \ref{HCST_EA_ON} and Lemma \ref{Con_in_EA}.\end{proof}

\begin{rem}
  Observe that in the proof of Theorem \ref{H^omega_nongodelian} we used only one direction of Lemma \ref{HCST_EA_ON}. Namely we employed the implication 
  $$\mathsf{EA}\vdash \varphi^{\mathcal{S}}\;\Rightarrow\; \mathsf{H}^{\omega}\vdash \mathcal{NAT}(\varphi),$$
  for $\Pi_1^{\mathsf{pred}}$ sentences $\varphi$. The inspection of the proof shows that in order to establish just this direction of Lemma \ref{HCST_EA_ON} it was possible to avoid the development of the bi-interpretation between $\mathsf{EA}$ and $\mathsf{EA}^{\mathsf{set}}$. And just develop the appropriate interpretation of $\mathsf{EA}$ in $\mathsf{EA}^{\mathsf{set}}$.  However, we consider Theorem \ref{bi-interpretability_theorem}, Lemma \ref{EA_EA^set_ON}, and Lemma \ref{HCST_EA_ON} to be interesting on their own merit and thus include them into the paper.
\end{rem}

\section{Theory $\mathsf{H}$ is non-G{\"o}delian} \label{H_appendix}
In this section we sketch the proof of the fact that theory $\mathsf{H}_{<\omega}$ proves its own consistency. The reasons of why it is the case are roughly speaking the same as for the case of the theory $\mathsf{H}^{\omega}_{<\omega}$. We dedicated the main part of the paper to the case of higher-order theory since the conservation results for $\mathsf{H}^{\omega}$ are in our opinion more appealing and since the development of ordinal arithmetic in $\mathsf{H}^{\omega}$  is more straightforward.  

Recall that the theory $\mathsf{H}$ is a first-order theory with equality which signature contains the binary membership predicate $\in$ and the unary function $\overline{\mathsf{V}}$. Axioms of $\mathsf{H}$:
\begin{enumerate}
    \item  $x=y \mathrel{\leftrightarrow} \forall z(z\in x\mathrel{\leftrightarrow} z\in y))$ (Extensionality);
\item $\exists y \forall z(z\in y\mathrel{\leftrightarrow}z\in x \land \varphi(z))$, where $\varphi$ range over formulas without free occurrences of the variable $y$ (Separation); 
\item $y\in \overline{\mathsf{V}}(x)\mathrel{\leftrightarrow}(\exists z\in x)(y\subseteq \overline{\mathsf{V}}(z))$ (Defining Axiom for $\overline{\mathsf{V}}$).
\end{enumerate}

We will start with proving the analogue of Lemma \ref{foundation_in_H^omega}. However for this we will need to work with classes in theory $\mathsf{H}$. We do it in the same style as in other first-order set theory: classes are collections of sets $\{x\mid \varphi(x)\}$, where $\varphi(x)$ is some first-order formula (possibly with parameters). In this approach, of course, we could not do quantifications over classes.

\begin{lem} Theory $\mathsf{H}$ prove any instances of the scheme of $\varepsilon$-induction:
$$\forall x((\forall y\in x) \varphi(y)\to \varphi(x))\to \forall x \;\varphi(x).$$\end{lem}
\begin{proof}We are going to adopt the proof of Lemma \ref{foundation_in_H^omega}. The only obstacle is that the least progressive class $\mathsf{WF}$ have been defined as the intersection of all progressive classes and this definition could not be directly mimicked in a first-order theory.   To address this we will give an alternative definition of $\mathsf{WF}$ (by a first-order formula), prove in $\mathsf{H}$ that $\mathsf{WF}$ is a progressive class, and prove in $\mathsf{H}$ a scheme of a theorem that if $C$ is a progressive class then $\mathsf{WF}\subseteq C$. With the use of this kind of definition of the class $\mathsf{WF}$ we could directly adopt other parts of the proof of Lemma \ref{foundation_in_H^omega} for the case of $\mathsf{H}$. 

 We define the class $\mathsf{WF}$ to be the intersection of the classes 
\begin{enumerate}
    \item $\mathsf{WF}^{\mathsf{pre}}_1$ consisting of of all $x$ such that $\overline{\mathsf{V}}(x)$ is transitive; 
    \item $\mathsf{WF}^{\mathsf{pre}}_2$ consisting of  of all $x$ such that $x\subseteq \overline{\mathsf{V}}(x)$;
    \item \label{wf_of_A} $\mathsf{WF}^{\mathsf{pre}}_3$ consisting of of all $x$ such that there is a $\in$-minimal element in any non-empty $y\subseteq \overline{\mathsf{V}}(x)$.
\end{enumerate}
 Using the defining axiom for $\overline{\mathsf{V}}$ it is easy to show that $\mathsf{WF}^{\mathsf{pre}}_1$, $\mathsf{WF}^{\mathsf{pre}}_2$, and $\mathsf{WF}^{\mathsf{pre}}_3$ are progressive classes. Thus $\mathsf{WF}$ is progressive. Let us finally show that $\mathsf{WF}$ is the least progressive class. We consider a progressive class $C$, set $x\in \mathsf{WF}$ and claim that $x\in C$. Since $x\subseteq \overline{\mathsf{V}}(x)$ and $C$ is progressive it would be enough to show that $\overline{\mathsf{V}}(x)\subseteq C$. To achieve the latter goal we assume for a contradiction that the set $\overline{\mathsf{V}}(x)\setminus C$ is non-empty. We know that any non-empty subset of $\overline{\mathsf{V}}(x)$ have a $\in$-minimal element. Let $y\in \overline{\mathsf{V}}(x)\setminus C$ be a $\in$-minimal element. We have $y\subseteq C$. The transitivity of the set  $\overline{\mathsf{V}}(x)$ implies that $y\subseteq C$. And by progressivity of $C$ we get $y\in C$, contradiction.\end{proof}

The standard definitions (as in Section \ref{ordinal_arithmetic_in_H_omega}) of classes $\mathsf{Trans},\mathsf{On},\mathsf{Nat}$, order $<$, and constant $0$ could be carried out in $\mathsf{H}$ in a standard fashion. However, we could not construct the functions of ordinal arithmetic in $\mathsf{H}$ exactly in the same way we have done it  in $\mathsf{H}^{\omega}$ or $\mathsf{EA}^{\mathsf{set}}$. Let us focus on the case of addition function, since multiplication and exponentiation functions could be defined in essentially the same way. Unlike $\mathsf{H}^{\omega}$, in $\mathsf{H}$ we do not have access to quantification over partial class-functions $+'\colon \mathsf{On}\times \mathsf{On}\to \mathsf{On}$. And unlike $\mathsf{EA}^{\mathsf{set}}$, in $\mathsf{H}$ we could not construct set-size partial function $+_{\delta}\colon \delta\times \delta\to \delta$: the set $+_{\delta}$ is constructed as a subset of $\mathcal{P}(\mathcal{P}(\mathcal{P}(\mathcal{P}(\delta))))$. 

However, it is possible to modify the approach base of the function $+_{\delta}$ to the case of $\mathsf{H}$. We consider the classes $\mathsf{V}^{-n}$: the class $\mathsf{V}^{-0}=\mathsf{V}$ and the class $\mathsf{V}^{-(n+1)}$ is the class of all sets $x$ such that $x\in y$ for some $y\in V^{-n}$. Due to the totality of $\overline{\mathsf{V}}$ function, for any $x\in \mathsf{V}^{-(n+1)}$ the powerset $\mathcal{P}(x)$ is defined and lies in $\mathsf{V}^{-n}$. We denote by $\mathsf{On}^{-n}$ the class $\mathsf{On}\cap\mathsf{V}^{-n}$ and by $\mathsf{Nat}^{-n}$ the class $\mathsf{Nat}\cap \mathsf{V}^{-n}$. By $\varepsilon$-induction we prove existence and uniqueness for all $\delta\in \mathsf{On}^{-4}$ of the partial set-size functions $+_{\delta}\colon \delta\times\delta\to \delta$ such that
\begin{enumerate}
\item $\alpha+_{\delta}\beta$ is defined iff $\alpha,\beta\in \delta$, for all $\gamma<\beta$ the value $\alpha+_\delta\gamma$ is defined, and $\sup(\{\alpha\}\cup\{S(\alpha+_\delta\gamma)\mid \gamma<\beta\})\in \delta$;
\item if $\alpha+_\delta\beta$ is defined then $\alpha+_\delta\beta=\sup(\{\alpha\}\cup\{S(\alpha+_\delta\gamma)\mid \gamma<\beta\})$.
\end{enumerate}
Next we define the class-size function $+^{-5}\colon \mathsf{On}^{-5}\times\mathsf{On}^{-5}\to \mathsf{On}^{-5}$ as the union of all partial addition functions $+_{\delta}$, for $\delta\in \mathsf{On}^{-4}$. Observe that the difference between $\mathsf{On}^{-5}$ and $\mathsf{On}$ is that there could be at most $5$ topmost ordinals that are in $\mathsf{On}$ but not in $\mathsf{On}^{-5}$. Hence by separate consideration of this topmost ordinals we could give a first-order definition of a partial addition function $+\colon \mathsf{On}\times \mathsf{On}\to \mathsf{On}$ such that for all $\alpha,\beta\in \mathsf{On}$ we have
$$\alpha+\beta=\sup(\{\alpha\}\cup \{S(\alpha+\gamma)\mid \gamma<\beta\}).$$ By $\varepsilon$-induction we prove that $+$ is the unique class function that satisfies this recursive definition.

After development of ordinal arithmetic in $\mathsf{H}$ we modify  the interpretation $\mathcal{NAT}$ with the definitions that work in $\mathsf{H}$ (in $\mathsf{H}^{\omega}$ it is easy to prove the equivalence of the definitions in the new version of $\mathcal{NAT}$ and the version of $\mathcal{NAT}$ from Section \ref{ordinal_arithmetic_in_H_omega}). Hence we have an embedding of predicate-only arithmetical language into $\mathsf{H}$. We also denote by $\mathcal{NAT}^{-n}$ the modification of the interpretation $\mathcal{NAT}$ with the domain of the interpretation being the class $\mathsf{Nat}^{-n}=\mathsf{Nat}\cap \mathsf{V}^{-n}$ rather than the class $\mathsf{Nat}$.

The following result is a version of Lemma \ref{HCST_EA_ON} for the case of $\mathsf{H}$:
\begin{lem} \label{H_and_EA_eq} Suppose $\varphi$ is a $\Pi_1^{\mathsf{pred}}$ sentence. Then
  $$\mathsf{EA}\vdash \varphi^{\mathcal{S}}\iff \mbox{ for some $n$ we have }\mathsf{H}\vdash \mathcal{NAT}^{-n}(\varphi).$$
\end{lem}
\begin{proof}
By Lemma \ref{EA_EA^set_ON} it will be enough to prove that 
$$\mathsf{EA}^{\mathsf{set}}\vdash \mathcal{ON}(\varphi)\iff \mbox{ for some $n$ we have }\mathsf{H}\vdash \mathcal{NAT}^{-n}(\varphi).$$

Assume $\mathsf{H}\vdash \mathcal{NAT}^{-n}(\varphi)$. Then we reason in $\mathsf{EA}^{\mathsf{set}}$ to prove $\mathcal{ON}(\varphi)$. The sentence $\varphi$ is of the form $\forall x_1,\ldots,x_k\;\psi(x_1,\ldots,x_k)$, where $\psi$ is $\Delta_0^{\mathsf{pred}}$.  We consider some ordinals $\alpha_1,\ldots,\alpha_k$ and claim that $\mathcal{ON}(\psi(\alpha_1,\ldots,\alpha_k))$. We consider the transitive set $M=\overline{\mathsf{V}}(\max(\alpha_1,\ldots,\alpha_k)+n+1)$. Observe that $\mathfrak{M}=(M,\in,\overline{\mathsf{V}})$ is a transitive model of $\mathsf{H}$ such that ordinal arithmetic inside $\mathfrak{M}$ (given by the definitions for the theory $\mathsf{H}$) coincide with the standard ordinal arithmetic restricted to $M$. And observe that all $\alpha_i$ are in the class $\mathsf{Nat}^{-n}$ of the model $\mathcal{M}$. By internalization of $\mathsf{H}$ proof of $\mathcal{NAT}^{-n}(\varphi)$ we get that $\mathfrak{M}\models \mathcal{NAT}^{-n}(\varphi)$ and hence $\mathfrak{M}\models \mathcal{NAT}(\psi(\alpha_1,\ldots,\alpha_k))$. Thus $\mathcal{ON}(\psi(\alpha_1,\ldots,\alpha_k))$.

Now assume that for all $n$ we have  $\mathsf{H}\nvdash \mathcal{NAT}^{-n}(\varphi)$. We are going to construct a model of $\mathsf{EA}^{\mathsf{set}}$, where $\mathcal{ON}(\varphi)$ fails. By compactness there is a model $\mathfrak{M}$ of $\mathsf{H}$ with $\mathfrak{M}$-naturals $\alpha_1,\ldots,\alpha_k$ such that $\mathfrak{M}\not\models \psi(\alpha_1,\ldots,\alpha_k)$ and  $\mathfrak{M}\not\models \alpha_i\in \mathsf{Nat}^{-n}$, for any $1\le i\le k$ and $n\ge 0$. We consider the intersection $N$ of all $\mathfrak{M}$-classes $\mathsf{Nat}^{-n}$ (for standard $n$). And next we consider the submodel $\mathfrak{N}$ of $\mathfrak{M}$ that consists of all $a\in \mathfrak{M}$ such that $\mathfrak{M}\models a\in \overline{\mathsf{V}}(\alpha)$, for some $\alpha\in N$. Using axiomatization of $\mathsf{EA}^{\mathsf{set}}$ from Lemma \ref{EA_set_alt} it is easy to see that $\mathfrak{N}$ is a model of $\mathsf{EA}^{\mathsf{set}}$. And from construction it is clear that $\mathfrak{N}\models \mathcal{ON}(\varphi)$.\end{proof}

For a rational $a>0$ we denote as $a \mathcal{S}$ the cut consisting of numbers $x$ such that $[x/a]\in \mathcal{S}$. 

\begin{cor}\label{H_and_EA_imp} Suppose $\varphi$ is a $\Pi_1^{\mathsf{pred}}$ sentence,  $\mathsf{EA}\vdash \varphi^{2\mathcal{S}}$. Then $\mathsf{H}\vdash \mathcal{NAT}(\varphi)$.
\end{cor}
\begin{proof} 
Let us define interpretation $\mathcal{SUM}$ of the predicate-only arithmetical language in itself such that numbers are interpreted by pairs of numbers $(x_1,x_2)$ which intended value is $x_1+x_2$. The key feature of $\mathcal{SUM}$ is that all the quantifiers in the translations of atomic formulas are bounded. This allows us to verify the properties of $\mathcal{SUM}$ in theory $\mathsf{H}$ that could not even prove totality of successor function.

The interpretation of equality $(x_1,x_2)=^{\star}(y_1,y_2)$ is $\Delta_0^{\mathsf{pred}}$ formula $$(\exists z\le y_1)(x_1+z=y_1\land y_2+z=x_2)\lor (\exists z\le y_2)(x_2+z=y_2\land y_1+z=x_1).$$
 The interpretation of comparison $(x_1,x_2)\le^{\star}(y_1,y_2)$ is $\Delta_0^{\mathsf{pred}}$ formula
$$\begin{aligned}(\exists z_1,z_2\le y_1)(z_1\le z_2\land x_1+z_2=y_1\land y_2+z_1=x_2)\lor\\(\exists z_1,z_2\le y_2)(z_1\le z_2\land x_2+z_2=y_2\land y_1+z_1=x_1).\end{aligned}$$
The interpretation of successor function graph $S((x_1,x_2))=^{\star}(y_1,y_2)$ is $\Delta_0^{\mathsf{pred}}$ formula
$$\begin{aligned}(\exists z_1,z_2\le y_1)(S(z_1)=z_2\land x_1+z_2=y_1\land y_2+z_1=x_2)\lor\\(\exists z_1,z_2\le y_2)(S(z_1)=z_2\land x_2+z_2=y_2\land y_1+z_1=x_1).\end{aligned}$$
The interpretation $(x_1,x_2)+(x_3,x_4)=^{\star}(y_1,y_2)$ of addition function graph is a $\Delta_0^{\mathsf{pred}}$ formula that expresses the fact that there exist splittings $x_1=z_{1,1}+z_{1,2}$, $x_2=z_{2,1}+z_{2,2}$, $x_3=z_{3,1}+z_{3,2}$, and $x_4=z_{4,1}+z_{4,2}$ such that $y_1=z_{1,1}+z_{2,1}+z_{3,1}+z_{4,1}$ and $y_2=z_{1,2}+z_{2,2}+z_{3,2}+z_{4,2}$. The interpretation $(x_{1,1},x_{1,2})(x_{2,1},x_{2,2})=^{\star}(y_1,y_2)$ of multiplication function graph is a $\Delta_0^{\mathsf{pred}}$ formula that expresses the fact that for all $1\le i,j\le 2$ there exists splitting $x_{i,j}=z_{i,j,1}+z_{i,j,2}$ such that
\begin{enumerate}
    \item $z_{1,i_1,j_1}z_{2,i_2,j_2}\le \max(y_1,y_2)$, for all $1\le i_1,i_2,j_1,j_2\le 2$;
    \item for $1\le i_1,i_2,j_1,j_2,k\le 2$ there exist $w_{i_1,i_2,j_1,j_2,k}\le y_k$ such that
    \begin{enumerate} 
      \item $w_{i_1,i_2,j_1,j_2,1}+w_{i_1,i_2,j_1,j_2,2}=z_{1,i_1,j_1}z_{2,i_2,j_2}$, for all $1\le i_1,i_2,j_1,j_2\le 2$ (we use the condition 1. above to make the comparison in a $\Delta_0^{\mathsf{pred}}$ formula),
      \item $\sum\limits_{1\le i_1,i_2,j_1,j_2\le 2} w_{i_1,i_2,j_1,j_2,k}=y_k$, for $1\le k\le 2$.
    \end{enumerate}
\end{enumerate}

Let us now explain how to define interpretation $2^{(x_1,x_2)}=^{\star}(y_1,y_2)$ of binary exponentiation function graph. We describe it as an algorithm that could be easily transformed to a $\Delta_0^{\mathsf{pred}}$-formula. In the case of $x_1=x_2=0$ we define $2^{(x_1,x_2)}=^{\star}(y_1,y_2)$ to be true if either $y_1=0$ and $y_2=S(0)$, or $y_1=S(0)$ and $y_2=0$. Otherwise, we consider $y_0=\max(y_1,y_2)$. We find $y_0\ge z$ such that $z=x_1+x_2$ (if there are no $z$ with this property then we put $2^{(x_1,x_2)}=^{\star}(y_1,y_2)$ to be false). We find $z'\le z$ such that $S(z')=z$. And we find $w\le y_0$ such that $w=2^z$ (if there are no $w$ with this property then we put $2^{(x_1,x_2)}=^{\star}(y_1,y_2)$ to be false). We put $2^{(x_1,x_2)}=^{\star}(y_1,y_2)$ to be true iff $(w,w)=^{\star}(y_1,y_2)$.   in the case $x_1+x_2\ge 1$  we use the fact that if  $2^{x_1+x_2}=y_1+y_2$ then $2^{x_1+x_2-1}\le \max(y_1,y_2)$ and $2^{x_1+x_2-1}+2^{x_1+x_2-1}=2^{x_1+x_2}$. 

One could check that both $\mathsf{EA}$ and $\mathsf{H}$ verify that the interpretation $\mathcal{SUM}$ works as intended, i.e. for any predicate symbol $P(x_1,\ldots,x_n)$ of predicate-only arithmetical language and its $\mathcal{SUM}$-interpretation $P^{\star}((x_{1,1},x_{1,2}),\ldots,(x_{n,1},x_{n,2}))$ we have
\begin{enumerate}
    \item $\begin{aligned}\mathsf{EA}\vdash &\mathop{\forall}\limits_{1\le i\le n,\; 0\le j\le 2} x_{i,j}(  \bigwedge\limits_{1\le i\le n} x_{i,1}+x_{i,2}=x_{i,0} \to \\ & (P(x_{1,0},\ldots,x_{n,0})\mathrel{\leftrightarrow} P^{\star}((x_{1,1},x_{1,2}),\ldots,(x_{n,1},x_{n,2})))),\end{aligned}$
    \item $\begin{aligned}\mathsf{H}\vdash \mathcal{NAT}(&\mathop{\forall}\limits_{1\le i\le n,\; 0\le j\le 2} x_{i,j}(  \bigwedge\limits_{1\le i\le n} x_{i,1}+x_{i,2}=x_{i,0} \to \\ & (P(x_{1,0},\ldots,x_{n,0})\mathrel{\leftrightarrow} P^{\star}((x_{1,1},x_{1,2}),\ldots,(x_{n,1},x_{n,2}))))).\end{aligned}$
\end{enumerate}

Recall that we consider a $\Pi_1^{\mathsf{pred}}$ sentence $\varphi$ such that $\mathsf{EA}\vdash \varphi^{2\mathcal{S}}$. Since $\mathsf{EA}$ verifies that the sums of pairs of numbers from $\mathcal{S}$ are precisely the numbers from $2\mathcal{S}$, we have $\mathsf{EA}\vdash (\mathcal{SUM}(\varphi))^{\mathcal{S}}$. And since the interpretation $\mathcal{SUM}$ interpretes all the predicate symbols by $\Delta_0^{\mathsf{pred}}$ formulas, the translation $\mathcal{SUM}(\varphi)$ is a $\Pi_1^{\mathsf{pred}}$ formula. Thus by Lemma \ref{H_and_EA_eq} we have $\mathsf{H}\vdash \mathcal{NAT}^{-n}(\mathcal{SUM}(\varphi))$, for some $n$. Now let us reason in $\mathsf{H}$ to prove $\varphi$. We consider two cases: 1. the number $2n$ does not exists, 2. the number $2n$ exists. In the case 1. we exploit the fact that $\varphi$ is a true $\Pi_1^{\mathsf{pred}}$ fact and just formalize in $\mathsf{H}$ the direct verification of the fact that $\varphi$ holds when the range of quantifiers is restricted to the numbers $<2n$. In the case 2. we observe that any natural number $x\in \mathsf{Nat}$ is equal to the sum $x_1+x_2$ for some $x_1,x_2\in \mathsf{Nat}^{-n}$ and hence  $\mathcal{NAT}^{-n}(\mathcal{SUM}(\varphi))$ implies $\varphi$.
\end{proof}
Note that with additional efforts, in Corollary \ref{H_and_EA_imp} the cut $2\mathcal{S}$ could be replaced with the cut $(1+\varepsilon)\mathcal{S}$, for any fixed rational $\varepsilon>0$.

An easy adaptation of Lemma \ref{Con_in_EA} to the case of $\mathsf{H}$ shows that $\mathsf{EA}\vdash (\mathsf{Con}^{\mathsf{pred}}(\mathsf{H}_{<\omega}))^{2\mathcal{S}}$. And by Corollary \ref{H_and_EA_imp} we conclude
\begin{thm} \label{Con_in_H}
$\mathsf{H}$ proves $\mathcal{NAT}(\mathsf{Con}^{\mathsf{pred}}(\mathsf{H}_{<\omega}))$.
\end{thm}

\appendix
\section{ $\mathsf{H}_{<\omega}$ and Robinson's Arithmetic $\mathsf{R}$} \label{interpretability}
In the section we present several observations about the connection between our theory $\mathsf{H}_{<\omega}$ and Robinson's arithmetic $\mathsf{R}$.  They were noticed by Albert Visser when he read a draft of this paper.

Robinson's arithmetic $\mathsf{R}$ is a weak arithmetical theory introduced by Tarski, Mostowski, and Robinson \cite{TMR53} and known to be hereditarily undecidable. As usual, the numeral $\underline{n}$ is the term $S^n(0)$. The axioms of Robinson's arithmetic $\mathsf{R}$ are:
\begin{enumerate}
\item $\underline{n}+\underline{m}=\underline{n+m}$;
\item $\underline{n}\cdot\underline{m}=\underline{n\cdot m}$;
\item $\underline{n}\ne \underline{m}$, for $n\ne m$;
\item $x\le n \to \bigvee\limits_{ i\le n} x=\underline{i}$;
\item $x\le \underline{n}\lor \underline{n}\le x$.
\end{enumerate}
Visser have showed that the interpretability class of the theory $\mathsf{R}$ is fairly special. A theory is called locally finitely satisfiable if any its finite subtheories have a finite model. The theory $\mathsf{R}$ is locally finitely satisfiable. And  it were proved by Visser \cite{Vis14} that any c.e. locally finitely satisfiable theory $T$ is interpretable in $\mathsf{R}$. In other words the interpretability class of $\mathsf{R}$ is the greatest among interpretability classes of c.e. locally finitely satisfiable theories.

Observe that theory $\mathsf{H}_{<\omega}$ is locally finitely satisfiable, thus it is interpretable in $R$. On the other hand, using our development of ordinal arithmetic in $\mathsf{H}$ we could interpret $\mathsf{R}$ in  $\mathsf{H}_{<\omega}$. The domain of the interpretation are finite ordinals. The constant $0$ is interpreted by the empty set. For the purposes of the interpretation we make the  partial functions $S,+,\cdot$ total by assigning the value $0$ to the inputs where partial functions were undefined. Trivial check shows that the translations of all the axioms of $\mathsf{R}$ are provable in $\mathsf{H}_{<\omega}$.

Thus $\mathsf{R}$ and $\mathsf{H}_{<\omega}$ are mutually interpretable. And hence theory $\mathsf{H}_{<\omega}$ from interpretability theoretic point of view could be regarded as the set-theoretic analogue of $\mathsf{R}$. However, unlike  $\mathsf{H}_{<\omega}$, to the best of the author's knowledge, no analogue of the self-verification property of $\mathsf{H}_{<\omega}$ is known for $\mathsf{R}$. And due to the extreme weakness of $\mathsf{R}$, I do not expect that $\mathsf{R}$ proves its own consistency for any natural arithmetization of its consistency assertion. 


\section*{Acknowledgments}
I am grateful to Lev Beklemishev for introducing me to the Willard's approach to construction of theories proving their own consistency\footnote{Note that Beklemishev have developed a still unpublished simplification of Willard's construction (which is different from the approach used in the present paper).} and for latter stimulating discussion of the results of the present paper. And I am grateful to Albert Visser for his useful comments (see Appendix \ref{interpretability}).

This work is supported in part by Young Russian Mathematics award.

\bibliographystyle{abbrv}
\bibliography{bibliography}

\end{document}